\documentclass[11pt, letterpaper]{amsart}

\usepackage{amssymb,amscd}
\usepackage{mathrsfs}
\usepackage{hyperref}
\usepackage[svgnames]{xcolor}
\hypersetup{colorlinks,breaklinks,
            linkcolor=DarkBlue,urlcolor=DarkBlue,
            anchorcolor=DarkBlue,citecolor=DarkBlue}
\usepackage{xy}
\usepackage{tikz}
\usetikzlibrary{decorations.pathreplacing,shapes.geometric}
\usetikzlibrary{calc,positioning}
\usepackage{graphicx}
\usepackage{euscript}
\usepackage[verbose,letterpaper,tmargin=1in,bmargin=1in,lmargin=1in,rmargin=1in,nomarginpar]
{geometry}

\newtheorem{theorem}{Theorem}[section]

\newtheorem{proposition}[theorem]{Proposition}
\newtheorem{corollary}[theorem]{Corollary}
\newtheorem{lemma}[theorem]{Lemma}

\theoremstyle{definition}
\newtheorem{definition}[theorem]{Definition}

\newtheorem{example}[theorem]{Example}
\newtheorem{remark*}[theorem]{}

\theoremstyle{remark}
\newtheorem{remark}[theorem]{Remark}

\newcommand{\R}{\mathbb R}
\newcommand{\C}{\mathbb C}

\newcommand{\tr}{\mathrm{tr}}

\newcommand{\ms}[1]{\mathscr{#1}}
\newcommand{\mc}[1]{\mathcal{#1}}
\newcommand{\mb}[1]{\mathbb{#1}}

\newcommand{\wh}[1]{\widehat{#1}}
\DeclareMathOperator{\ev}{ev}

\newcommand{\jw}{\mathbin{
\begin{tikzpicture}[scale=.208,baseline=0.21pt]
 \draw[line width=0.45] (0,1) --++(0.15,0) arc(90:45:0.25)--(0.6,0.5);
 \draw[line width=0.45] (1.2,0) --++(-0.15,0) arc(270:225:0.25)--(0.6,0.5);
 \draw[line width=0.45] (1.2,1) --++(-0.15,0) arc(90:135:0.25)--(0.6,0.5);
 \draw[line width=0.45] (0,0) --++(0.15,0) arc(-90:-45:0.25)--(0.6,0.5);
 \end{tikzpicture}}
}

\newcommand{\twostrings}{\begin{tikzpicture}[scale=.208,baseline=0.21pt]
 \draw[line width=0.45] (0,0) --++(1,0);
 \draw[line width=0.45] (0,1)   --++(1,0);
 \end{tikzpicture}}

\newcommand{\jwcup}{\mathbin{
\begin{tikzpicture}[scale=.208,baseline=0.21pt]
 \draw[line width=0.45] (0,0) --++(0.25,0) arc(-90:-45:0.25) --++(.75,.75) arc(-45:0:.25) --++(0,.25);
  \draw[line width=0.45,xscale=-1] (-1.75,0) --++(0.25,0) arc(-90:-45:0.25) --++(.75,.75) arc(-45:0:.25) --++(0,.25);

 \end{tikzpicture}}
 }
\newcommand{\sidecup}{\mathbin{
\begin{tikzpicture}[scale=.208,baseline=0.21pt]
 \draw[line width=0.45] (1.2,1) --++(-0.25,0) arc(90:270:0.25 and 0.5) --++(0.25,0);
 \end{tikzpicture}}
 }
 
 \newcommand{\jones}{
 {\mathbin{
 \begin{tikzpicture}[scale=.208,baseline=0.21pt,xscale=-1]
 \draw[line width=0.45] (1.2,1) --++(-0.25,0) arc(90:270:0.25 and 0.5) --++(0.25,0);
 \end{tikzpicture} \hskip 1pt
\begin{tikzpicture}[scale=.208,baseline=0.21pt]
 \draw[line width=0.45] (1.2,1) --++(-0.25,0) arc(90:270:0.25 and 0.5) --++(0.25,0);
 \end{tikzpicture}}}}

\numberwithin{equation}{section}

\tikzset{Box/.style={very thick, rounded corners}}
\tikzset{marked/.style={star, star point height = .75mm, star points =5, fill=black,minimum size=2mm, inner sep=0mm} }
\tikzset{verythickline/.style = {line width=7pt}}
\tikzset{thickline/.style = {line width=5pt}}
\tikzset{medthick/.style = {line width=3pt}}
\tikzset{med/.style = {line width=2pt}}
\tikzset{count/.style = {fill=white,circle,draw,thin, inner sep=2pt}}
\tikzset{rcount/.style = {fill=white,rectangle,draw,thin,inner sep=2pt, rounded corners}}
\tikzset{cpr/.style = {draw,fill=white,rectangle,thin, rounded corners}}

\newcommand{\dcup}{\tikz[scale=.5]{\draw (0,.5) -- ++(0,-.25) arc(180:360:.25) -- ++(0,.25); \draw (.1,.5) -- ++(0,-.25) arc(180:360:.15) -- ++(0,.25);}}
\newcommand{\corner}{\tikz{\draw (0,0) arc(-90:0:.15 and .25); \draw (.5,0) arc(270:180:.15 and .25);}}
\begin{document}

\title{Free analysis and planar algebras}
\author[S. Curran]{S. Curran$^\star$}\thanks{ $\star$: Research supported by an NSF postdoctoral fellowship and DARPA Award HR0011-12-1-0009.}
\address{S.C.: D.E. Shaw Group, 1166 Avenue of the Americas, 9th Floor, New York NY 10036}
\email{\href{mailto:steve.curran218@gmail.com}{steve.curran218@gmail.com}}
\author[Y. Dabrowski]{Y. Dabrowski$^\dagger$}\thanks{\noindent $\dagger$: Partially supported by ANR grant NEUMANN}\address{Y.D.: Universit\'e de Lyon, Universit\'e Lyon 1, Institut Camille Jordan UMR 5208, 43 blvd. du 11 novembre 1918, F-69622 Villeurbanne cedex, France.}
\email{\href{mailto:dabrowski@math.univ-lyon1.fr}{dabrowski@math.univ-lyon1.fr}}
\author[D. Shlyakhtenko]{D. Shlyakhtenko$^\ddagger$}\thanks{\noindent $\ddagger$: Research supported by NSF grant DMS-1161411  and DARPA Award HR0011-12-1-0009.}
\address{D.S.: Department of Mathematics, UCLA, Los Angeles, CA 90095.}
\email{\href{mailto:shlyakht@math.ucla.edu}{shlyakht@math.ucla.edu}}

\begin{abstract}
We study 2-cabled analogs of Voiculescu's trace and free Gibbs states on Jones planar algebras.  These states are traces on a tower of graded algebras associated to a Jones planar algebra. Among our results is that, with a suitable definition, finiteness of free Fisher information for planar algebra traces implies that the associated tower of von Neumann algebras consists of factors, and that the standard invariant of the associated inclusion is exactly the original planar algebra.  We also give conditions that imply that the associated von Neumann algebras are non-$\Gamma$ non-$L^2$ rigid factors.
\end{abstract}

\maketitle

\section*{Introduction}
A recent series of papers \cite{gjs1,gjs2,cjs,bro1} investigated towers of von Neumann algebras associated to a Jones planar algebra. To such a planar algebra, one first associates a sequence of algebras $Gr_k$.  Next, a special trace (the Voiculescu trace $\tau$) is constructed on each of these algebras.  It turns out that $W^*(Gr_k,\tau)$ are factors (in fact, interpolated free group factors), and the standard invariant  of their Jones tower is the original planar algebra $\mc P$ (compare with Popa's results \cite{popa}).
A question dating back to \cite{gjs1} is whether there are other ``nice'' choices of traces on $Gr_k$; in particular, one is interested in questions such as factoriality of $W^*(Gr_k,\tau)$ for various choices of $\tau$, as well as the computation of the standard invariant of the resulting tower of algebras.

A step in studying more general traces was taken in \cite{gjsz}, where it was shown that certain random matrix models lead to ``free Gibbs states'' on planar algebras.  Such a free Gibbs state corresponds to a certain trace $\tau_k$ on each $Gr_k$, and these traces have interesting combinatorial properties related to questions of enumeration of planar maps.

The question of the isomorphism class of algebras associated to some of these free Gibbs states has been recently settled by B. Nelson \cite{nelson} using his non-tracial extension \cite{nelson:qf} of free monotone transport introduced in \cite{gsh}.  He showed that if the potential of the free Gibbs state is sufficiently close to the Voiculescu trace (in the sense that the potential is sufficiently close to the quadratic one), then the von Neumann algebras generated by $Gr_k$ under the free Gibbs state and the Voiculescu trace are isomorphic (and are thus interpolated free group factors).

In this paper we investigate a different collection of traces, among which a special role is played by a trace which turns out to be as canonical as Voiculescu's trace.  We call the trace the 2-cabled Voiculescu trace.  The main difference from the Voiculescu trace is the replacement of summation over all Templerley-Lieb diagrams with summations over 2-cabled Temperley-Lieb diagrams.  
The relationship between the 2-cabled Voiculescu trace and the usual Voiculescu is akin to the difference between circular and semicircular systems.  Not surprisingly, many of the results from \cite{gjs1,gjs2} can be reproved for the 2-cabled version.  We sketch the proofs of several of these. Among these results is the existence of 2-cabled versions of free Gibbs states.

It turns out that the 2-cabled setup is nicely amenable to using tools from free probability theory.  Voiculescu's free analysis differential calculus has a nice diagrammatic expression adapted to this situation.  The diagrammatic calculus is technically better behaved than the one used in the 1-cabled situation \cite{gjsz,nelson}.   Using ideas of \cite{dabrowski:Fisher} (which relied on some techniques of J. Peterson \cite{peterson:ell2rigidity}), we are able to prove that if the free Fisher information of a planar algebra trace is finite, then the von Neumann algebras generated by $Gr_k$ are factors, and the standard invariant of the resulting Jones tower is again the planar algebra $P$.  This is in particular the case for any (two-cabled) free Gibbs state (whether or not the potential is close to quadratic).  A consequence of our work is that possible phase transitions phenomena arising when one changes the potential of a free Gibbs state away from quadratic potentials cannot be captured by a change in the standard invariant. 

The outline of the paper is as follows.  In the following section we recall some background material on graded algebras associated to planar algebras from \cite{gjs1,gjs2,cjs}.  In Section \ref{sec:freeprob} we develop some general theory for traces on the graded algebra of a planar algebra.  In particular we introduce planar algebra cumulants, which generalize the free cumulants of Speicher, and construct some new examples of planar algebra traces which generalize the Voiculescu trace studied in \cite{gjs1,gjs2,cjs}.  In \S\ref{sec:freeAnalysis}, we adapt Voiculescu's free differential calculus to the planar algebra setting.  Our main result states that if the planar algebra version of Voiculescu's free Fisher information of a planar algebra trace is finite, then the tower of algebras associated to this trace has the same standard invariant as the original planar algebra.  This is in particular the case when the trace we consider is the free Gibbs state associated to a potential.   In this case, we also show that the associated von Neumann algebras are non-$\Gamma$, non $L^2$-rigid and prime.

In \S\ref{sec:2cabletrace} we further analyze the 2-cabled Voiculescu trace, and establish the isomorphism classes of the associated von Neumann algebras.  In \S\ref{sec:gibbs} we consider free Gibbs states on planar algebras, which are 2-cabled versions of those studied in \cite{gjsz}.  Finally, \S\ref{sec:transport} constructs a  free monotone transport between the 2-cabled Voiculescu trace and free Gibbs states with potential sufficiently close to the quadratic potential $\frac{1}{2}\dcup$.  The proof is essentially the same as that of B. Nelson \cite{nelson}.  In particular, we show that $Gr_k$ under such free Gibbs states still generate free group factors. 

\section{Graded algebras associated to planar algebras} 

\subsection{Background}
\label{sec:background}

We begin by briefly recalling some constructions from \cite{gjs1}, \cite{cjs}.

Let $\mc P = (P_n)_{n \geq 0}$ be a subfactor planar algebra.  For $n,k \geq 0$ let $P_{n,k}$ be a copy of $P_{n+k}$.  Elements of $P_{n,k}$ will be represented by diagrams
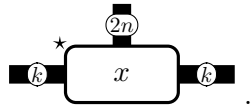
\begin{equation*}
 \begin{tikzpicture}[scale=.75]
 \draw[Box](0,0) rectangle (2,1);
  \draw[verythickline](0,.5) --node[rcount,scale=.75]{$k$}++ (-1,0); 
  \draw[verythickline](2,.5) --node[rcount,scale=.75]{$k$}++ (1,0);
  \draw[verythickline](1,1)  --node[rcount,scale=.75]{$2n$}++ (0,.75);
 \node at (1,.5) {$x$}; \node[marked] at (-.1,1.05) {};
 \end{tikzpicture}.\end{equation*}
In this diagrams and in diagrams below thick lines represent several parallel strings; sometimes, we add a numerical label to indicate their number (if the numeral is absent, the number of lines is presumed to be arbitrary, or understood from context).  Thus in this diagram, the thick lines to the left and right each represent $k$ strings, and the thick line at top represents $2n$ strings.  We will typically suppress the marked point $\star$, and take the convention that it occurs at the top-left corner which is adjacent to an unshaded region.

Define a product $\wedge_k:P_{n,k} \times P_{m,k} \to P_{n+m,k}$ by
\begin{equation*}
\begin{tikzpicture}[thick,scale=.75]
\draw[Box](0,0) rectangle (2,1); \node[marked,above left,scale=.9] at (0,1) {};
  \draw[verythickline](0,.5) -- (-.5,.5); \draw[verythickline](2,.5) -- (2.5,.5);
 \draw[verythickline](1,1) -- (1,1.25);
 \node at (1,.5) {$x$}; \node[left] at (-.7,.5) {$x \wedge_k y =$};
\begin{scope}[xshift=2.9cm]
\draw[Box](0,0) rectangle (2,1); \node[marked,above left,scale=.9] at (0,1) {};
  \draw[verythickline](0,.5) -- (-.5,.5); \draw[verythickline](2,.5) -- (2.5,.5);
 \draw[verythickline](1,1) -- (1,1.25);
 \node at (1,.5) {$y$};
\end{scope}
\end{tikzpicture}
\end{equation*}
The involution $\dagger:P_{n,k} \to P_{n,k}$ is given by
\begin{equation*}
\begin{tikzpicture}[thick,scale=.75]
\draw[Box](0,0) rectangle (2,1);
  \draw[verythickline](0,.5) -- (-.5,.5); \draw[verythickline](2,.5) -- (2.5,.5);
 \draw[verythickline](1,1) -- (1,1.25);
 \node at (1,.5) {$x^\dagger$}; \node at (3,.5) {$=$};
\begin{scope}[xshift=4cm]
\draw[Box](0,0) rectangle (2,1);
  \draw[verythickline](0,.5) -- (-.5,.5); \draw[verythickline](2,.5) -- (2.5,.5);
 \draw[verythickline](1,1) -- (1,1.25);
 \node at (1,.5) {$x^*$};
 \draw[Box] (-.5,-.25) rectangle (2.5,1.25); \node[marked] at (2.1,1.05) {}; \node[marked] at (-.6,1.3) {};
\end{scope}
\end{tikzpicture}
\end{equation*}
We then define the $*$-algebra $(Gr_k(\mc P), \wedge_k, \dagger)$,
\begin{equation*}
 Gr_k(\mc P) = \bigoplus_{n \geq 0} P_{n,k}.
\end{equation*}
The unit of $Gr_k(\mc P)$ is the element of $P_{0,k}$ consisting of $k$ parallel lines. We have unital inclusions $Gr_k(\mc P) \hookrightarrow Gr_{k+1}(\mc P)$ determined by
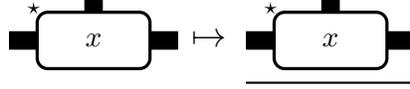
\begin{equation*}
 \begin{tikzpicture}[scale=.75]
 \draw[Box](0,0) rectangle (2,1);  \node[marked,above left,scale=.9] at (0,1) {};
  \draw[verythickline](0,.5) -- (-.5,.5); \draw[verythickline](2,.5) -- (2.5,.5);
\draw[verythickline](1,1) -- (1,1.25);
 \node at (1,.5) {$x$};
\begin{scope}[xshift=4.2cm]
 \draw[Box](0,0) rectangle (2,1);   \node[marked,above left,scale=.9] at (0,1) {};
  \draw[verythickline](0,.5) -- (-.5,.5); \draw[verythickline](2,.5) -- (2.5,.5);
\draw[verythickline](1,1) -- (1,1.25);
\draw[thick] (-.5,-.25) -- (2.5,-.25);
 \node at (1,.5) {$x$};
\node[left=1mm,scale=1.2] at (-.5,.5) {$\mapsto$};
\end{scope}
 \end{tikzpicture}
\end{equation*}

\subsubsection{Popa's symmetric enveloping algebra}
For integers $k,s,t$ with $s+t + 2k = n$, let $V_{k}(s,t)$ be a copy of $P_{n}$.  Elements of $V_{k}(s,t)$ will be represented by diagrams of the form
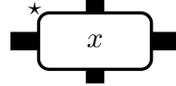
\begin{equation*}
 \begin{tikzpicture}[scale=.75]
 \draw[Box](0,0) rectangle (2,1);
  \draw[verythickline](0,.5) -- (-.5,.5); \draw[verythickline](2,.5) -- (2.5,.5);
 \draw[verythickline](1,0) -- (1,-.25); \draw[verythickline](1,1) -- (1,1.25);
 \node at (1,.5) {$x$}; \node[marked, above left] at (0,1) {};
 \end{tikzpicture}
\end{equation*}
where there are $2s$ parallel strings at the top, $2t$ strings at the bottom and $2k$ at either side.   As above, we will use the convention that the marked point occurs at the upper left corner, which is adjacent to a unshaded region. 

Define a product $\wedge:V_{k}(s,t) \times V_{k}(s',t') \to V_{k}(s+s',t+t')$ by
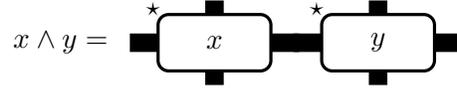
\begin{equation*}
\begin{tikzpicture}[thick,scale=.75]
\draw[Box](0,0) rectangle (2,1);\node[marked, above left] at (0,1) {};
  \draw[verythickline](0,.5) -- (-.5,.5); \draw[verythickline](2,.5) -- (2.5,.5);
 \draw[verythickline](1,0) -- (1,-.25); \draw[verythickline](1,1) -- (1,1.25);
 \node at (1,.5) {$x$}; \node[left] at (-.7,.5) {$x \wedge y =$};
\begin{scope}[xshift=2.9cm]
\draw[Box](0,0) rectangle (2,1);\node[marked, above left] at (0,1) {};
  \draw[verythickline](0,.5) -- (-.5,.5); \draw[verythickline](2,.5) -- (2.5,.5);
 \draw[verythickline](1,0) -- (1,-.25); \draw[verythickline](1,1) -- (1,1.25);
 \node at (1,.5) {$y$};
\end{scope}
\end{tikzpicture}
\end{equation*}
The adjoint $\dagger:V_{k}(s,t) \to V_{k}(s,t)$ is defined by
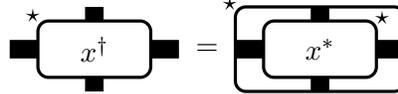
\begin{equation*}
\begin{tikzpicture}[thick,scale=.75]
\draw[Box](0,0) rectangle (2,1);\node[marked, above left] at (0,1) {};
  \draw[verythickline](0,.5) -- (-.5,.5); \draw[verythickline](2,.5) -- (2.5,.5);
 \draw[verythickline](1,0) -- (1,-.25); \draw[verythickline](1,1) -- (1,1.25);
 \node at (1,.5) {$x^\dagger$};
\node at (3,.5) {$=$};
\begin{scope}[xshift=4cm]
\draw[Box](0,0) rectangle (2,1); 
  \draw[verythickline](0,.5) -- (-.5,.5); \draw[verythickline](2,.5) -- (2.5,.5);
 \draw[verythickline](1,0) -- (1,-.25); \draw[verythickline](1,1) -- (1,1.25);
 \node at (1,.5) {$x^*$}; \node[marked] at (2.1,1.05) {};
\draw[Box](-.5,-.25) rectangle (2.5,1.25); \node[marked] at (-.6,1.3) {};
\end{scope}
\end{tikzpicture} 
\end{equation*}
Popa's symmetric enveloping graded algebra $(Gr_k(\mc P) \boxtimes Gr_k(\mc P)^{op}, \wedge, \dagger)$ is defined by
\begin{equation*}
 Gr_k(\mc P) \boxtimes Gr_k(\mc P)^{op} = \bigoplus_{s,t \geq 0} V_{k}(s,t).
\end{equation*}
There is a natural inclusion $Gr_k(\mc P) \otimes Gr_k(\mc P)^{op} \hookrightarrow Gr_k(\mc P) \boxtimes Gr_k(\mc P)^{op}$ determined by
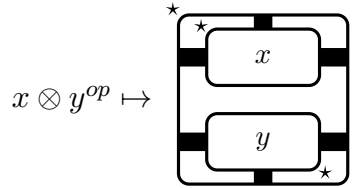
\begin{equation*}
 \begin{tikzpicture}[scale=.75]
 \draw[Box] (0,0) rectangle (2,1); \node at (1,.5) {$x$};
 \draw[verythickline] (-.5,.5) -- (0,.5); \draw[verythickline] (2,.5) -- (2.5,.5);
 \draw[verythickline] (1,1) -- (1,1.25); \node[marked] at (-.1,1.05) {};
\begin{scope}[rotate around={180:(1,.5)}, yshift=1.5cm]
 \draw[Box] (0,0) rectangle (2,1); \node at (1,.5) {$y$}; \node[marked] at (-.1,1.05) {};
 \draw[verythickline] (-.5,.5) -- (0,.5); \draw[verythickline] (2,.5) -- (2.5,.5);
 \draw[verythickline] (1,1) -- (1,1.25);
\end{scope}
\draw[Box] (-.5,-1.75) rectangle (2.5,1.25); \node[marked] at (-.6,1.35) {};
\node[left=2mm,scale=1.1] at (-.5,-.25) {$x \otimes y^{op}  \mapsto$};
 \end{tikzpicture}
\end{equation*}


\subsection{An embedding of graph planar algebras.}  Let $\Gamma = \Gamma_+ \cup \Gamma_-$ be a (unoriented) connected bipartite graph.  For an edge $e$ let $s(e)$ (resp. $t(e)$) denote the initial (resp. terminal) vertex of $e$, let $e^o$ denote the ``opposite edge'' from $t(e)$ to $s(e)$, and let $E_+$ denote the set of edges with $s(e) \in \Gamma_+$.  Let $\mu = (\mu_v)_{v \in \Gamma}$ be a positive eigenvector of the adjacency matrix of $\Gamma$ with eigenvalue $\delta > 0$ (e.g. the unique Perron-Frobenius eigenvector if $\Gamma$ is finite).  With this data Jones \cite{grpalg} has constructed a planar algebra $\mc P^\Gamma = (P_m)_{m \geq 0}$ with $P_0 = \C\Gamma_+$ and $P_m$ the vector space with basis given by loops $(e_1,\dotsc,e_{2m})$ with $s(e_1) =t(e_{2m}) \in \Gamma_+$.  See \cite{gjs1} for the description of the action of planar tangles, note that we use the convention on ``spin factors'' from there, which differs slightly from the original definition in \cite{grpalg}.  Any subfactor planar algebra $\mc P$ 
can be embedded 
as a planar subalgebra of $\mc P^\Gamma$ for some $\Gamma$ (in particular one may take $\Gamma$ to be the principal graph of $\mc P$, see \cite{jpen}).

The graded algebras associated to $\mc P^\Gamma$ can be embedded into some natural algebras arising from the graph, as we will now discuss.

\begin{definition} Define $\ms A_\Gamma$ to be the ``even part'' of the path algebra of $\Gamma$, i.e. the $*$-algebra with generators $\{X_{e,f^o}: e,f \in E_+, \; t(e) = t(f)\} \cup \{p_v:v \in \Gamma_+\}$ subject to the relations
 \begin{itemize}
 \item $p_v$ are mutually orthogonal projections.
 \item $p_v X_{e,f^o} p_w = \delta_{v=s(e)}\delta_{w=s(f)} \cdot X_{e,f^o}$.
 \item $X_{e,f^o}^* = X_{f,e^o}$.
 \end{itemize}

\end{definition}

\begin{proposition}\label{graph-ids}
We have the following identifications:
\begin{enumerate}
 \item  $Gr_0(\mc P^\Gamma)$ can be identified with the subalgebra of $\ms A_\Gamma$ spanned by ``loops'', i.e.
\begin{equation*}
 Gr_0(\mc P^\Gamma) \simeq \sum_{v \in \Gamma_+} p_v\ms A_\Gamma p_v = \C \Gamma_+ \oplus \bigoplus_{m \geq 1} \mathrm{span}\langle X_{e_1,e_2}\dotsb X_{e_{2m-1},e_{2m}}: s(e_1) = t(e_{2m})\rangle.
\end{equation*}
For loops $(e_1,f_1^o,\dotsc,e_m,f_m^o)$ in $P_m \subset Gr_0(\mc P^\Gamma)$, the identification is given by
\begin{equation*}
(e_1,f_1^o,\dotsc,e_m,f_m^o) \Leftrightarrow (\mu_{s(e_1)}\mu_{s(f_1)}\dotsb \mu_{s(e_m)}\mu_{s(f_m)})^{-1/4} \cdot X_{e_1,f_1^o} \dotsb X_{e_m,f_m^o}.
\end{equation*}

\item Let $N = \{(e,f^o): e,f \in E_+, \; t(e) = t(f)\}$, then there is a linear identification of $Gr_1(\mc P^\Gamma)$ with the subspace of $(\ms A_\Gamma)^N$ consisting of $Y = (Y_{e,f^o})$ such that $Y_{e,f^o} = p_{s(e)}Y_{e,f^o}p_{s(f)}$.  This is determined by identifying a loop $(e_1,f_1^o,\dotsc,e_{m+1},f_{m+1}^o) \in P_{m,1} = P_{m+1}$ with the vector $Y = (Y_{e,f^o})$ given by 
\begin{equation*}
Y_{f_{m+1},e_{m+1}^o} = \frac{\mu_{s(f_{m+1})}}{(\mu_{s(e_1)}\mu_{s(f_1)}\dotsb \mu_{s(e_{m+1})}\mu_{s(f_{m+1})})^{1/4}}\cdot X_{e_1,f_1^o}\dotsb X_{e_m,f_m^o}
\end{equation*}
and $Y_{e,f^o} = 0$ unless $e = f_{m+1}$ and $f=e_{m+1}^o$.
\item $Gr_1(\mc P^\Gamma) \boxtimes Gr_1(\mc P^\Gamma)^{op}$ can be identified with the compression $p\bigl[M_N(\ms A_\Gamma \otimes \ms A_{\Gamma}^{op})\bigr]p$, where $p$ is the diagonal matrix with $(e,f^o),(e,f^o)$-entry equal to $p_{s(e)}\otimes p_{s(f)}$.   This is determined by identifying a loop $(e_1,f_1^o,\dotsc,e_{l+m+2},f_{l+m+2}^o) \in V_2(l,m) = P_{l + m + 2}$ with
\begin{multline*}
\frac{\mu_{s(f_{l+m+2})}}{(\mu_{s(e_1)}\mu_{s(f_1)}\dotsb \mu_{s(e_{l+m+2})}\mu_{s(f_{l+m+2})})^{1/4}}\\
\times p_{s(f_{m+l+2})}X_{e_1,f_1^o}\dotsb X_{e_{l},f_l^o} \otimes (p_{s(f_{l+1})}X_{e_{l+2},f_{l+2}^o}\dotsb X_{e_{l+m+1},f_{l+m+1}^o})^{op} \otimes V_{(f_{l+m+2}^o,e_{l+m+2}),(e_{l+1},f_{l+1}^o)}
\end{multline*}
where $(V_{(e,f^o),(g,h^o)})$ is a system of matrix units for $M_N(\C)$.
\end{enumerate}

\end{proposition}

\begin{proof}
These identifications follow directly from the definition of $\mc P^\Gamma$, see \cite{gjs1}.
\end{proof}

\begin{remark}\label{normalize}
For the identification (1) above one can think of the formula
\begin{equation*}
 \begin{tikzpicture}
  \draw[Box] (0,.5) rectangle (1,1); \node[below] at (.3,.9) {$e$}; \node[below] at (.7,1) {$f^o$};\node[marked,above left,scale=.9] at (0,1) {};
  \draw[thick] (.3,1) --++(0,.25); \draw[thick] (.7,1) -- ++(0,.25);
  \node[left] at (-.3,.75) {$ = (\mu_{s(e)}\mu_{s(f)})^{1/4} \cdot$};
    \node[right] at (1.3,.75) {$\in P_1$};
  \begin{scope}[xshift=-4.5cm]
   \draw[Box] (0,.5) rectangle (1,1); \node[scale=.9] at (.5,.75) {$X_{e,f^o}$}; \node[marked,above left,scale=.9] at (0,1) {};
  \draw[thick] (.3,1) --++(0,.25); \draw[thick] (.7,1) -- ++(0,.25);
  \end{scope}
 \end{tikzpicture}
\end{equation*}
Note however that $(e,f^o)$ will not be a true element of $P_1$ unless $s(e) = s(f)$.  However $X_{e,f^o}X_{f,e^o}$ is always an element of $P_2$, and the normalization is chosen such that
\begin{equation*}
\begin{tikzpicture}
\draw[Box] (0,.5) rectangle (1,1); \node[scale=.9] at (.5,.75) {$X_{e,f^o}$}; \node[marked,above left,scale=.9] at (0,1) {};
\draw[Box] (1.5,.5) rectangle (2.5,1); \node[scale=.9] at (2,.75) {$X_{f,e^o}$}; \node[marked,above left,scale=.9] at (1.5,1) {};
\draw[thick] (.7,1) arc(180:90:.2) -- (1.6,1.2) arc(90:0:.2);
\draw[thick] (.3,1) arc(180:90:.4) -- (1.8,1.4) arc(90:0:.4);
\node[right] at (2.8,.75) {$= \mu_{s(f)} \cdot p_{s(e)} \in P_0 = \C\Gamma_+.$};
\end{tikzpicture}
\end{equation*}

\end{remark}

%
%

\section{Free probability and planar algebras}\label{sec:freeprob}

Let $\mc P$ be a subfactor planar algebra, and let $\tau_0$ be a linear functional on $Gr_0(\mc P)$.  By duality, there are elements $T_m \in P_m$ such that
\begin{equation*}
  \begin{tikzpicture}[scale=.5]
   \draw[Box] (0,0) rectangle (2,1); \node at (1,.5) {$x$}; \node[marked,scale=.8,above left] at (0,1) {};
   \draw[Box] (0,1.5) rectangle (2,2.5); \node at (1,2) {$T_m$}; \node[marked,scale=.8,below right] at (2,1.5) {};
   \draw[thickline] (1,1) -- ++(0,.5);
   \node[left] at (-.5,1.25) {$\tau_0(x) = $};
   \node[right] at (3,1.25) {($x \in P_m$)};
  \end{tikzpicture}\end{equation*}
We can then extend $\tau_0$ to $\tau_k:Gr_k(\mc P) \to \C$ by 
\begin{equation*}
 \begin{tikzpicture}[scale=.5]
   \draw[Box] (0,0) rectangle (2,1); \node at (1,.5) {$x$}; \node[marked,scale=.8,above left] at (0,1) {};
   \draw[Box] (0,1.5) rectangle (2,2.5); \node at (1,2) {$T_m$}; \node[marked,scale=.8,below right] at (2,1.5) {};
   \draw[thickline] (1,1) -- ++(0,.5);
   \draw[thickline] (0,.5) arc(90:270:.5) -- ++(2,0) arc(-90:90:.5);
   \node[left] at (-.75,1) {$\tau_k(x) = \delta^{-k} \cdot $};
   \node[right] at (3.25,1) {($x \in P_{m,k}$)};
  \end{tikzpicture}
\end{equation*}

View $P_k$ as a subalgebra of $Gr_k(\mc P)$ by viewing an element $z\in P_k$ as an element of $Gr_k(\mc P)$ (with no vertical strings). Then $\tau_0$ also defines a conditional expectation ${\mc E}_k : Gr_k(\mc P) \to P_k$ by the formula
\begin{equation*}
  \begin{tikzpicture}[scale=.5]
   \draw[Box] (0,0) rectangle (2,1); \node at (1,.5) {$x$}; \node[marked,scale=.8,above left] at (0,1) {};
   \draw[Box] (0,1.5) rectangle (2,2.5); \node at (1,2) {$T_m$}; \node[marked,scale=.8,below right] at (2,1.5) {};
   \draw[thickline] (1,1) -- ++(0,.5);
   \draw[thickline] (0,0.5) -- ++(-0.5, 0);
      \draw[thickline] (2,0.5) -- ++(0.5, 0);
   \node[left] at (-.5,1.25) {${\mc E}_k(x) = $};
   \node[right] at (3,1.25) {($x \in P_{m,k}$)};
  \end{tikzpicture}\end{equation*}
It is not hard to see that $$\tau_k(x) = \tau_k({\mc E}_k(x)),\qquad x\in Gr_k(\mc P).$$

We also define $\tau_k \boxtimes \tau_k^{op}: Gr_k \boxtimes Gr_k^{op} \to \C$ by
\begin{equation*}
\begin{tikzpicture}[scale=.5]
  \draw[Box] (0,0) rectangle (2,1); \node at (1,.5) {$x$}; \node[marked,scale=.8,above left] at (0,1) {};
  \draw[thickline] (0,.5) arc(90:180:.5cm) --++(0,-1.25) arc(180:270:.5) --++(2,0) arc(-90:0:.5) --++(0,1.25) arc(0:90:.5);
  \draw[Box] (0,-.25) rectangle (2,-1.25); \node at (1,-.75) {$T_s$}; \node[marked,scale=.8,above left] at (0,-.25) {};
  \draw[Box] (0,1.25) rectangle (2,2.25); \node at (1,1.75) {$T_t$}; \node[marked,scale=.8,below right] at (2,1.25) {};
  \draw[verythickline](1,0) -- (1,-.25); \draw[verythickline] (1,1) -- (1,1.25);
  \node[left] at (-.7,.5) {$(\tau_k \boxtimes \tau_k^{op})(x) = \delta^{-2k} \cdot$};
  \node[right] at (4,.5) {$\bigl(x \in V_k(s,t)\bigr)$};
 \end{tikzpicture}
\end{equation*}

Note that the various inclusions ($Gr_k \subset Gr_{k+1}$, $Gr_k\otimes Gr_k^{op}\subset Gr_k\boxtimes Gr_k^{op}$ and so on) are trace-preserving with these definitions of traces.

We will sometimes identify $\tau = (\tau_k)_{k \geq 0}$ with $(T_m)_{m \geq 0}$.  

\begin{definition} We say that $\tau$ is a (faithful) positive $\mc P$-trace if $\tau_k$ is a (faithful) positive tracial state for each $k$.  We say that $\tau$ is bounded if  $Gr_k(\mc P)$ acts by bounded operators on the GNS Hilbert space. 
\end{definition}

\begin{proposition} Let $k,l$ be arbitrary and $n\geq \max(k,l)$.  Define $$
 R(n,\tau)_{k,l}  = \delta^{-(n-k)-(n-l)}
  \begin{array}{c}
\begin{tikzpicture}[scale=.5]
    \draw[thickline] (0,00) -- node[rcount, scale=0.75] {$k$} ++(1.25,0) arc(90:180:-.5cm) --++(0,0.25); 
    \draw[thickline] (0,-3) -- node[rcount, scale=0.75] {$k$} ++(1.25,0) --++(0.5,0)  arc(90:180:-.5cm) --++(0,3.25); 
   \draw[thickline] (0,-2.5) arc(90:180:-0.75) -- node[rcount, scale=0.75] {$n-k$} ++  (0,0.5) arc(180:270:-0.75);

    \draw[thickline] (5,00) -- node[rcount, scale=0.75] {$\ l\ $} ++(-1.25,0) arc(90:0:-.5cm) --++(0,0.25); 
    \draw[thickline] (5,-3) -- node[rcount, scale=0.75] {$\ l\ $} ++(-1.25,0) --++(-0.5,0)  arc(90:0:-.5cm) --++(0,3.25); 
      \draw[thickline] (5,-2.5) arc(90:0:-0.75) -- node[rcount, scale=0.75] {$n-l$} ++  (0,0.5) arc(180:90:.75);
    \draw[Box](1,2.25) rectangle (4,0.75); \node at (2.5,1.5){$T_{2k+2l}$};

 \end{tikzpicture}
  \end{array} \in P_{2n}\subset Gr_{2n}(\mc P).
  $$
 Then the $\mc P$-trace $\tau$ is positive iff for each $n$, $m$ and arbitrary $k_1,\dots,k_m \leq n$, the matrix $$R(\tau) = (R(n,\tau)_{k_i,k_j})_{ij}\in M_{m\times m}((Gr_{2n}(\mc P)) $$ is positive.
\end{proposition}
\begin{proof}
Assume that we are given an element $x\in Gr_s(\mc P)$, so that $x=\sum x_i$ with $x_i \in P_{k_i+s}$.
Consider the following elements: $$
y_i = \delta^{-3s/2} \begin{array}{c}
\begin{tikzpicture}[scale=.5]
   \draw[Box](0,0) rectangle (2,1); \node at (1,0.5){$x_i$};
   \node[marked,scale=.8,above left] at (0,1) {};
   
   \draw[thickline](0.75,1) --++(0,1.25) arc(180:90:.25)  --node[rcount, scale=0.75] {$k_i$} ++(3.5,0);
   \draw[thickline](1.25,1) --++(0,.25) arc(180:90:.25)  --node[rcount, scale=0.75] {$k_i$} ++(1,0)
   	arc(90:0:.25) --++(0,-2) arc(180:90:-.25) --++(-3.75,0);
	
  \draw[thickline](2,0.5) arc(90:-90:0.25 and 0.5) --node[rcount, scale=0.75] {$\ s\ $}++(-3.25,0);
  \draw[thickline](0,0.5) --node[rcount, scale=0.75] {$\ s\ $}++(-1.25,0);
   \draw[thickline](-1.25,-1.75) -- node[rcount, scale=0.75] {$n-k_i$} ++(4.25,0) arc(-90:0:0.25) 
     --++(0,3) arc(180:90:.25)--++(1,0);
   \draw[thickline] (4.5,0.25) arc(90:180:.5) --node[rcount, scale=0.75] {$\ s\ $}++(0,-1) arc(0:90:-0.5);
 \end{tikzpicture}
 \end{array}\in Gr_{n+2s},\qquad
 c = 
 \begin{array}{c}
 \begin{tikzpicture}[scale=.5]
 \draw[thickline] (-3,-1.5) arc(90:180:-.5)  --node[rcount, scale=0.75] {$n+2s$}++(0,.5)  arc(180:270:-0.5);
 \draw[thickline] (0,0) arc(90:180:.5) --node[rcount, scale=0.75] {$n+2s$}++(0,-.5) arc(0:90:-0.5);
    \end{tikzpicture}
\end{array}\in Gr_{2n+4s}.
$$
Then if we denote by $\iota_{u,u+v}: Gr_{u} \to Gr_{u+v} $ the embedding map that adds $v$ horizontal strings at the bottom of the diagram, we see that 
\begin{multline*}
\sum_{ij} c \wedge \iota_{n+2s,2n+4s}(y_i) 
 \wedge  \iota_{2n+2s,2n+4s}(R(n+s,\tau)_{k_ik_j})\wedge \iota_{n+2s,2n+4s} (y_j^*) \wedge  c  = \tau_s ( x \wedge_s x^* ) c .
\end{multline*}
  Thus if the matrix $R = (R(n+s,\tau)_{k_i,k_j})_{ij}$ is positive, we deduce that $\tau_s (xx^*)\geq 0$ (since $\iota_{u,u+v}$, being an embedding of $C^*$-algebras, is completely-positive).  

Conversely, let $z_i \in P_{2n} \subset Gr_{2n}$ be arbitrary.  Since the restriction of $\tau_{2n}$ to $P_{2n}\subset Gr_{2n}$ is positive (this restriction does not depend on $\tau$),  it follows that $R(\tau)$ is positive iff $\tau_{2n} ( z_i R_{i,j} z_j^* )$ is a positive matrix for all choices of $z_i$.  
 
On the other hand, $$\delta^{2n} \tau_{2n}(z_i R_{i,j} z_j^*) = \delta^{-(n-k_i)-(n-k_j)}
 \begin{array}{c}
\begin{tikzpicture}[scale=.5]
\draw[Box](-1.5,-4) rectangle (-0.5,1);\node at (-1,-1.5){$z_i$};
	 \node[marked,scale=.8,above left] at (-1.5,1) {};

    \draw[thickline] (-.5,00) -- node[rcount, scale=0.75] {$k_i$} ++(1.75,0) arc(90:180:-.5cm) --++(0,0.25); 
    \draw[thickline] (-0.50,-3) -- node[rcount, scale=0.75] {$k_i$} ++(1.75,0) --++(0.5,0)  arc(90:180:-.5cm) --++(0,3.25); 
   \draw[thickline] (-0.5,-2.5) --++(0.5,0) arc(90:180:-0.75) -- node[rcount, scale=0.75] {$2n-2k_i$} ++  (0,0.5) arc(180:270:-0.75) --++(-.5,0) ;

    \draw[thickline] (5.5,00) -- node[rcount, scale=0.75] {$k_j$} ++(-1.75,0) arc(90:0:-.5cm) --++(0,0.25); 
    \draw[thickline] (5.5,-3) -- node[rcount, scale=0.75] {$k_j$} ++(-1.75,0) --++(-0.5,0)  arc(90:0:-.5cm) --++(0,3.25); 
      \draw[thickline] (5.5,-2.5)  --++(-0.5,0) arc(90:0:-0.75) -- node[rcount, scale=0.75] {$2n-2k_j$} ++  (0,0.5) arc(180:90:.75) --++(0.5,0);
\draw[Box](5.5,-4) rectangle (6.5,1); \node[xscale=-1] at (6,-1.5){$z_j$}  ;
	 \node[marked,scale=.8,above right] at (6.5,1) {};
    \draw[Box](1,2.25) rectangle (4,0.75); \node at (2.5,1.5){$T_{2k+2l}$};
    \draw[verythickline] (6.5,-0.25) --node[rcount, scale=1] {$n$}++ (1.5,0) arc (90:240:-1) --++(-3,1)
    	arc(-100:-90:-2) --++(-5.25,0) arc(90:100:2) --++(-3,-1) arc(120:270:1) --node[rcount, scale=1] {$n$}++(1.5,0);
    \draw[verythickline] (6.5,-2.75) --node[rcount, scale=1] {$n$}++ (1.5,0) arc (-90:-270:-1) --++(-11,0) arc(-90:-270:1) --node[rcount, scale=1] {$n$}++(1.5,0);
 \end{tikzpicture}
  \end{array}.
  $$ This can be redrawn as:
  $$
  \begin{array}{c}
  \begin{tikzpicture}[scale=.5]
  \draw[Box](0,0) rectangle (5.5,1);\node[rotate=90] at(2.75,0.5){$z_i$}; 
  \node[marked,scale=.8,above left] at (0,1) {};
  
  \draw[Box] (8,0) rectangle (13.5,1);\node  [rotate=-90,xscale=-1] at (10.75,0.5){$z_j$};
  \node[marked,scale=.8,above right] at (13.5,1) {};
  
  \draw[verythickline] (5.5,0.5)--node[rcount, scale=1] {$n$}++(2.5,0);
  \draw[verythickline] (0,0.5)  arc(-90:90:-1) --node[rcount, scale=1] {$n$}++(13.5,0) arc(-90:90:1);
  \draw[thickline] (1,1) arc(180:90:0.5 and 1) --node[rcount,scale=0.75]{$2n-2k_i$}++(2.5,0) arc(90:-0:0.5 and 1);
   \draw[thickline] (9,1) arc(180:90:0.5 and 1) --node[rcount,scale=0.75]{$2n-2k_j$}++(2.5,0) arc(90:-0:0.5 and 1);
   \draw[thickline](0.5,1) --node[rcount,scale=0.75]{$k_i$}++(0,2);
   \draw[thickline](5,1) --node[rcount,scale=0.75]{$k_i$}++(0,2);
   \draw[thickline](8.5,1) --node[rcount,scale=0.75]{$k_j$}++(0,2);
   \draw[thickline](13,1) --node[rcount,scale=0.75]{$k_j$}++(0,2);

   \draw[Box](0,3) rectangle (13.5,4.5);\node at(6.75,3.75){$T_{2k+2l}$};
   
   \end{tikzpicture}
   \end{array} = \delta^n \tau_n (z'_i (z'_j)^*)
$$
  where 
$$
z'_i =  \delta^{-(n-k_i)} \begin{array}{c}
  \begin{tikzpicture}[scale=.5]
  \draw[Box](0,0) rectangle (5.5,1);\node[rotate=90] at(2.75,0.5){$z_i$}; 
  \node[marked,scale=.8,above left] at (0,1) {};

  \draw[verythickline] (5.5,0.5)--node[rcount, scale=1] {$n$}++(1.5,0);
  \draw[verythickline] (0,0.5)  --node[rcount, scale=1] {$n$}++(-1.5,0);
  
  \draw[thickline] (1,1) arc(180:90:0.5 and 1) --node[rcount,scale=0.75]{$2n-2k_i$}++(2.5,0) arc(90:-0:0.5 and 1);

   \draw[thickline](0.5,1) --node[rcount,scale=0.75]{$k_i$}++(0,2);
   \draw[thickline](5,1) --node[rcount,scale=0.75]{$k_i$}++(0,2);
   
   \end{tikzpicture}
   \end{array}.
 $$

Thus if $\tau_n$ is positive for all $n$, it follows that the matrix $[\tau_n(z'_i (z'_j)^*)]_{ij}$ is positive, which in turn implies that $R(\tau)$ is positive.
\end{proof}

\begin{corollary}  \label{cor:embeddingTracePositivity} Assume that $\mc P\subset \mc P'$ is an embedding of planar algebras. Let $\tau$ be a positive $\mc P$-trace.  Extend $\tau$ (by the same diagrams) to a trace $\tau'$ on  $\mc P'$.  Then $\tau'$ is a positive $\mc P'$-trace.
\end{corollary}
\begin{proof}
Since $\mc P\subset \mc P'$ is a planar algebra embedding, it follows that $P_k \subset P'_k$ (regarded as subalgebras of $Gr_k(\mc P)$ and $Gr_k(\mc P')$) is an embedding of $C^*$-algebras.  Since $\tau$ is positive, the associated matrix $R$ is positive.  But then the matrix $R'$ associated to $\tau'$ must also be positive (being the image of the matrix $R$ under an inclusion of $C^*$-algebras).  Thus $\tau'$ is also positive.
\end{proof}
\subsection{Graph planar algebras}

Now let $\Gamma$ be the principal graph of $\mc P$, and recall that we have a planar algebra embedding $\mc P \hookrightarrow \mc P^\Gamma$. We can extend $\tau_k$ to a linear map $E_k: Gr_k(\mc P^\Gamma) \to l^\infty(\Gamma_+)$ by using the same diagrammatic formula.  Moreover, if $x \in Gr_k(\mc P)$ then we have $E_k(x) = \tau_k(x) \cdot 1$, where $1$ is the identity function in $l^\infty(\Gamma_+)$ (see \cite{gjs1}).  Since $\mc P\subset \mc P^\Gamma$ is a planar algebra embedding, we see that $\tau_k$ are positive  for all $k$ iff $E$ are positive for all $k$.

As an application, we record the following:

\begin{lemma}
Assume that $\tau_k$ are positive for all $k$.  Then $\tau_k\boxtimes\tau_k$ is also positive for all $k$.
\end{lemma}
\begin{proof}
By Corollary~\ref{cor:embeddingTracePositivity}, we may replace $\mc P$ by a graph planar algebra containing it.  We thus assume that $\mc P$ is the graph planar algebra of some graph $\Gamma$.  

Let $\gamma_i = (g^i_1,\cdots,g^i_{2s_i})$, $i=1,\dots,N$ be a finite collection of closed paths in $\Gamma$, and let $k$ be fixed. Let $t_i$ be fixed, and let us write $\gamma_i$ as the 
concatenation of six  paths: $$
\gamma_i = u_i \circ l_i  \circ l'_i \circ b_i \circ r_i \circ r'_i,$$  with
$u_1 = (g^i_1,\dots,g^i_{t_i})$, $l_i = (g^i_{t_i+1},\dots,g^i_{t_i+k})$,
$l'_i = (g^i_{t_i+k+1},\dots,g^i_{t_i+2k})$,
 $b_i = (g^i_{t_i+2k+1},\dots,g^i_{2s_i-2k})$, $r_i = (g^i_{2s_i-2k+1},\dots,g^i_{2s_i-k})$ and $r_i' = (g^i_{2s_i-k+1},\dots,g^i_{2s_i})$. 
 Consider the element
$$
x_i =\begin{array}{c}
 \begin{tikzpicture}[scale=.5]
 \draw[Box](0,0) rectangle (4,4);
\node[below] at (2,4) {$u_i$};
\node[below,rotate=-90] at (4,1){$l'_i$};\node[below,rotate=-90] at (4,3){$l_i$};
\node[below,rotate=180] at (2,0) {$b_i$};
\node[below,rotate=90] at (0,3) {$r'_i$};\node[below,rotate=90] at (0,1) {$r_i$};

\draw[thickline](2,4)-- node[rcount,scale=.75]{$t_i$}++(0,2);
\draw[thickline](2,0)-- node[rcount,scale=.75]{$2s_i-t_i-4k$}++(0,-2);

\draw[thickline](4,3)-- node[rcount,scale=.75]{$k$}++(2,0);
\draw[thickline](4,1)-- node[rcount,scale=.75]{$k$}++(2,0);

\draw[thickline](0,3)-- node[rcount,scale=.75]{$k$}++(-2,0);
\draw[thickline](0,1)-- node[rcount,scale=.75]{$k$}++(-2,0);

\node[marked,above left] at (0,4){};
\end{tikzpicture}\end{array} \in Gr_k\boxtimes Gr_k^{op}.
$$
Choose a path $w_i$ from the end of $l_i$ to the start of $r_i'$; because the graph is bipartite, $w_i$ has even length, say $2d_i$.  Let $d=\max_i d_i$.  
Let
$$
T_i = \begin{array}{c} 
\begin{tikzpicture}[scale=.5]
 \draw[Box](0,2) rectangle (4,4);
 \node[below] at (2,4) {$u_i$};
\node[below,rotate=-90] at (4,3){$l_i$};
\node[below,rotate=90] at (0,3) {$r'_i$};
\node[marked,above left] at (0,4){};

\draw[thickline](2,4)-- node[rcount,scale=.75]{$t_i$}++(0,2);

\draw[thickline](4,3)-- node[rcount,scale=.75]{$k$}++(2,0);

\draw[thickline](0,3)-- node[rcount,scale=.75]{$k$}++(-2,0);

\node[below,rotate=180] at (2,2) {$w_i$};
\draw[thickline] (2,2) arc(0:-90:0.5) -- node[rcount,scale=0.75]{$2d_i$} ++ (-3.5,0);
\draw[thickline] (6,2) --node[rcount,scale=0.75]{$d$} ++ (-1,0) arc (90:270:0.5 and 1.25) --++(1,0);
\draw[thickline] (-2,0.5) --node[rcount,scale=0.75]{$d-d_i$} ++(3,0) arc(90:-90:0.5) --++(-3,0);

\end{tikzpicture}\end{array}\in Gr_{k+2d},\qquad S_i = \begin{array}{c} 
\begin{tikzpicture}[scale=.5]
 \draw[Box](0,0) rectangle (4,2);
\node[below,rotate=180] at (2,0) {$b_i$};
\node[below,rotate=-90] at (4,1){$l'_i$};
\node[below,rotate=90] at (0,1) {$r_i$};
\node[marked,below right] at (4,0){};

\draw[thickline](2,0)-- node[rcount,scale=.75]{$t_i$}++(0,-2);

\draw[thickline](4,1)-- node[rcount,scale=.75]{$k$}++(2,0);

\draw[thickline](0,1)-- node[rcount,scale=.75]{$k$}++(-2,0);

\node[below] at (2,2) {$w_i$};
\draw[thickline] (2,2) arc(0:90:0.5) -- node[rcount,scale=0.75]{$2d_i$} ++ (-3.5,0);
\draw[thickline] (6,4.5) --node[rcount,scale=0.75]{$d$} ++ (-1,0) arc (90:270:0.5 and 1.25) --++(1,0);
\draw[thickline] (-2,4.5) --node[rcount,scale=0.75]{$d-d_i$} ++(3,0) arc(90:-90:0.5) --++(-3,0);

\end{tikzpicture}\end{array} \in Gr_{k+2d}^{op}.
$$

Finally, let $$Q = 
\begin{array}{c} 
\begin{tikzpicture}[scale=.5]
 
\draw[thickline] (0,0) --node[rcount,scale=0.75]{$k$} ++(7,0);
\draw[thickline] (0,-1) --node[rcount,scale=0.75]{$2d$} ++(2,0) arc(90:-90:0.5) --++(-2,0);
\draw[thickline] (7,-1) --node[rcount,scale=0.75]{$2d$} ++(-2,0) arc(90:270:0.5) --++(2,0);
\draw[thickline] (0,-3) --node[rcount,scale=0.75]{$k$} ++(7,0);

\end{tikzpicture}\end{array}\in Gr_{k+2d}\boxtimes Gr_{k+2d}^{op},\qquad Q'=
\begin{array}{c} 
\begin{tikzpicture}[scale=.5]
\draw[thickline] (0,-1) --node[rcount,scale=0.75]{$2d$} ++(2,0) arc(90:-90:0.5) --++(-2,0);
\draw[thickline] (6,-1) --node[rcount,scale=0.75]{$2d$} ++(-2,0) arc(90:270:0.5) --++(2,0);
\end{tikzpicture}\end{array}\in Gr_{4d}.
$$

Then there exist nonzero real constants $\lambda_i$ so that $$ (\tau_{k+2d}\boxtimes \tau_{k+2d}^{op})\left( Q \wedge ( T_i \otimes S_i) \wedge (T_j \otimes S_j)^\dagger \wedge Q \right) = \lambda_i \lambda_j (\tau_k\boxtimes \tau_k^{op}) (x_i \wedge x_j^\dagger).$$

Since $\tau_k$ is positive, we conclude that $\tau_k\otimes \tau_k^{op}$ is also positive, and therefore the matix whose $i,j$-th entry is given by
\begin{align*}
 \lambda_i^{-1}\lambda_j^{-1} 
( \tau_{k+2d}\boxtimes \tau_{k+2d}^{op}) & \left( Q \wedge ( T_i \otimes S_i) \wedge (T_j \otimes S_j)^\dagger \wedge Q \right) 
 \\ & = \lambda_i^{-1}\lambda_j^{-1} Tr_{4d} \left( Q' \wedge \left[ (\tau_k\otimes 1\otimes \tau_k^{op})\left( (T_i\otimes S_j )\wedge (T_j \otimes S_j)^\dagger \right) \right] \wedge Q'\right)
\end{align*}  
is also positive.  We conclude that $\tau_k\boxtimes\tau_k^{op}$ is positive.
\end{proof}

\begin{lemma}
If there is a $l^\infty(\Gamma_+)$-valued von Neumann probability space $(M,E:M \to l^\infty(\Gamma_+))$ with an embedding $\Phi:Gr_0(\mc P^\Gamma) \to M$ such that $E_0 = E \circ \Phi$, then $\tau$ is a positive and bounded $\mc P$-trace.
\end{lemma}

\begin{proof}
For $k \geq 0$ we can extend $\Phi$ to an embedding $\Phi_k:Gr_k(\mc P^\Gamma) \to M \otimes M_{n_k}(\C)$ for some $n_k$, such that $E_k = (E \otimes \tr) \circ \Phi_k$ (this follows from the description of $\mc P^\Gamma$, see \cite{gjs1}).  Fix any vertex $v \in \Gamma_+$ and let $\varphi$ be any positive (semi-definite) trace on $l^\infty(\Gamma)$.  As remarked above we have $\tau_k = \varphi \circ E_k$, and it follows that $\tau_k$ is positive. 
\end{proof}

\subsection{ Planar algebra cumulants:}

Let $\mc P = (P_m)_{m \geq 0}$ be a planar algebra and let $\rho = (\rho_m)_{m \geq 1}$ be a sequence with $\rho_m \in P_m$.  Define $\rho_\pi \in P_m$, for $\pi$ a non-crossing partition in $NC(m)$, inductively as follows: if $\pi = 1_m$ is the partition with one block, then $\rho_{\pi} = \rho_m$.  Otherwise, let $V = \{l+1,\dotsc,l+s\}$ be an interval of $\pi$ and define
\begin{equation*}
\begin{tikzpicture}[scale=.7]
 \draw[Box] (0,0) rectangle (2,1); \node at (1,.5) {$\rho_\pi$}; \node[marked,scale=.8,below right] at (2,0) {};
 \draw[thickline] (1,0) --++(0,-.5);
 \begin{scope}[xshift=4cm,yshift=-.25cm]
  \draw[Box] (-.25,1) rectangle (2.25,2); \node at (1,1.5) {$\rho_{\pi \setminus V}$}; \node[marked,scale=.8,below right] at (2.25,1) {};
  \draw[Box] (.5,0) rectangle (1.5,.75); \node at (1,.375) {$\rho_s$}; \node[marked,scale=.7,below right] at (1.5,0) {};
  \draw[thickline] (.25,1) --node[rcount, near end] {$2l$} ++(0,-2); \draw[thickline] (1.75,1) --++(0,-2);
  \draw[thickline] (1,0) --++(0,-1);
  \node[left] at (-.5,.25) {$=$};
 \end{scope}
 \end{tikzpicture}
\end{equation*}
where $\pi \setminus V$ is the partition obtained by removing the block $V$ from $\pi$.  

\begin{example}
For $\pi = \{\{1,4,6\},\{2,3\},\{5\} \in NC_5$ we have
\begin{equation*}
\begin{tikzpicture}[thick]
\begin{scope}[xscale=2]
\draw (0,.75) --++(1.25,0);
\draw (.25,.5) --++(.25,0);
\foreach \x in {0,.75,1.25} {\draw (\x,0) --++(0,.75);}
\foreach \x in {.25,.5,1} {\draw (\x,0) --++(0,.5);}
\node[left] at (-.25,.375) {$\pi = $};
\end{scope}
\begin{scope}[xshift=5.25cm,yshift=-.5cm]
\draw[Box] (0,1.25) rectangle (4.75,1.75); \node at (2.375,1.5) {$\rho_3$}; \node[marked,scale=.8,below right] at (4.75, 1.25) {};
\draw[Box] (.75,.5) rectangle (2,1); \node at (1.375,.75) {$\rho_2$}; \node[marked, scale=.8, below right] at (2,.5) {};
\draw[Box] (3,.5) rectangle (3.75,1); \node at (3.375,.75) {$\rho_1$}; \node[marked, scale=.8, below right] at (3.75,.5) {};

\foreach \x in {1,1.25,1.5,1.75,3.25,3.5} {\draw (\x,.5) --++(0,-.5);}
\foreach \x in {.25,.5,2.25,2.5,4.25,4.5} {\draw (\x,1.25) --++(0,-1.25);}
\node[left] at (-.5,.875) {$\mapsto \; \; \rho_\pi = $};
\end{scope}
\end{tikzpicture}
\end{equation*}
\end{example}

Given $x \in P_m,k$ we define
\begin{equation*}
 \begin{tikzpicture}[scale=.5]
  \draw[Box] (0,0) rectangle (1,1); \node at (.5,.5) {$x$}; \node[marked,scale=.6, above left] at (0,1) {};
  \draw[Box] (0,1.5) rectangle (1,2.5); \node at (.5,2) {$\rho_\pi$}; \node[marked,scale=.6,below right] at (1,1.5) {};
  \draw[thickline] (.5,1) --++(0,.5);
  \draw[thickline] (1,.5)--++(0.25,0);
  \draw[thickline] (0,.5)--++(-0.25,0);

  \node[left] at (-.5,1.25) {$\rho_\pi[x] = $};
 \end{tikzpicture}
\end{equation*}

\begin{definition}\label{planar_cumulants}
Let $\tau = (T_m)_{m \geq 1}$ be a $\mc P$-trace.  The \textit{planar algebra free cumulants} $\kappa_m^\mc P \in P_m$ are determined by the requirement
\begin{equation*}
 T_m = \sum_{\pi \in NC(m)} \kappa_\pi^\mc P.
\end{equation*}
\end{definition}

As with the usual free cumulants, we can solve for $\kappa_m^\mc P$ using M\"{o}bius inversion:
\begin{equation*}
 \kappa_m^\mc P = \sum_{\sigma \in NC(m)} \mu(\sigma,1_m) \cdot T_\sigma,
\end{equation*}
where $\mu(\sigma,\pi)$ is the M\"{o}bius function on the lattice $NC(m)$.

We also have the usual free cumulants $(\kappa^{Gr_k}_m)_{m \geq 1}$ associated to the noncommutative probability space $(Gr_k(\mc P),E_k : Gr_k(\mc P)\to P_k)$ (see e.g. \cite{nsp}).  If $x_1,\dotsc,x_m \in P_k$ then it is not too hard to see that
\begin{equation*}
\kappa_m^{\mc P}[x_1 \wedge_k x_2 \wedge_k \dotsb \wedge_k x_m] = \kappa^{Gr_k}_m(x_1,x_2,\dotsc,x_m).
\end{equation*}
In general such products don't span $P_m$, however we can still recover the usual free cumulants from the planar algebra cumulants by adapting the product formula (see \cite[Theorem 11.12]{nsp}).  We need the following notation: given $m_1 + \dotsb + m_n = m$, and $\pi \in NC(n)$, define $\wh \pi \in NC(m)$ by partitioning the $n$ blocks $\{1,\dotsc,m_1\}, \{m_1+1,\dotsc,m_1 + m_2\},\dotsc,\{m_1 + \dotsb + m_{n-1} + 1,\dotsc,m\}$ according to $\pi$.

\begin{proposition} \label{product_cumulants}
Let $x_1,\dotsc,x_n \in Gr_k(\mc P)$ with $x_i \in P_{m_i}$, and let $m = m_1 + \dotsb + m_n$.  Then we have
\begin{equation*}
 \kappa^{Gr_k}_n(x_1, x_2,\dotsc, x_ n) = \sum_{\substack{\pi \in NC(m)\\ \pi \vee \wh{0}_m = 1_m}} \kappa_\pi^{\mc P}[x_1 \wedge_k x_2 \wedge_k \dotsb \wedge_k x_n].
\end{equation*}
\end{proposition}

\begin{proof}
We have
\begin{align*}
\kappa^{Gr_k}_n[x_1,\dotsc,x_n] &= \sum_{\pi \in NC(n)} \mu(\pi,1_n) \tau_\pi[x_1,\dotsc,x_n] \\
&= \sum_{\pi \in NC(n)} \mu(\pi,1_n) T_{\wh \pi}(x_1 \wedge_0 x_2 \wedge_0 \dotsb \wedge_0 x_n)\\
&= \sum_{\substack{\pi \in NC(m)\\ \wh 0_m \leq \pi \leq 1_m}} \mu(\pi,1_m) T_{\pi}(x_1 \wedge_0 x_2 \wedge_0 \dotsb \wedge_0 x_n)
\end{align*}
and the result then follows by M\"{o}bius inversion in the lattice $NC(m)$.
\end{proof}

As a corollary we have the following analogue of Speicher's characterization of freeness by vanishing of mixed cumulants.
\begin{corollary}\label{cumulant_freeness}
Suppose that $(\mc A_i)_{i \in I}$ are subalgebras of $Gr_k(\mc P)$, and suppose that for any $a_1,a_2$ with $a_j \in \mc A_{i_j}$, $i_1 \neq i_2$, and for any $m$ we have
\begin{equation*}
\begin{tikzpicture}[thick,scale=.5]
\draw[Box] (-0.5,0) rectangle (1.5,1); \node at (.75,.5) {$a_1$}; \node[marked,scale=.8,left] at (-0.5,1) {};
\draw[Box] (3.5,0) rectangle (5.5,1); \node at (4.5,.5) {$a_2$}; \node[marked,scale=.8,left] at (3.5,1) {};
\draw[Box] (1.25,2) rectangle (3.75,3); \node at (2.5,2.5) {$\kappa_m^\mc P$}; \node[marked,scale=.8,above] at (2.5,3.15) {};

\draw (.35,1) .. controls ++(0,.65) and ++(0,-.65) .. (2.05,2);
\draw (.65,1) .. controls ++(0,.5) and ++(0,-.75) .. (2.35,2);
\draw (4.35,1) ..controls ++(0,.5) and ++(0,-.75) .. (2.65,2);
\draw (4.65,1) ..controls ++(0,.65) and ++(0,-.65) .. (2.95,2);

\draw[medthick] (0,1) arc(0:180:.6 and .25 ) --++(0,-1.75);  
\draw[medthick] (1.2,1) arc(180:0:.55 and .25) --++(0,-1.75);
\draw[medthick] (3.8,1) arc(0:180:.55 and .25) --++(0,-1.75);
\draw[medthick] (5,1) arc(180:0:.6 and .25) --++(0,-1.75);

\draw[medthick] (5.5,0.5) arc(90:0:.20 and .25) --node[rcount, scale=.75] {$k$} ++(0,-1);  
\draw[medthick] (-0.5,0.5) arc(90:180:.20 and .25) --node[rcount, scale=.75] {$k$} ++(0,-1);

\draw[medthick] (1.5,0.5) arc(90:0:.20 and .25) --node[rcount, scale=.75] {$k$} ++(0,-1);\draw[medthick] (3.5,0.5) arc(90:180:.20 and .25) --node[rcount, scale=.75] {$k$} ++(0,-1);

\draw[medthick] (1.75,2) arc(0:-180:.375 and .25) --++(0,1.25); 
\draw[medthick] (3.25,2) arc(180:360:.375 and .25) --++(0,1.25);

\node[right] at (6.5,1.5) {$= 0,$};
\end{tikzpicture}
\end{equation*}
where each thick line represents an arbitrary number of parallel strings (except for strings labeled with $k$, which represent exactly $k$ strings).  Then $(\mc A_i)_{i \in I}$ are free with amalgamation over $P_k$.  In particular, if $k=0$, they are free with respect to $\tau_0$.
\end{corollary}

\begin{proof}
Let $a_j \in \mc A_{i_j}$ for $j = 1,\dotsc,m$, and suppose that not all $i_j$ are equal.  By the proposition we have
\begin{equation*}
 \kappa_m(a_1,\dotsc,a_m) = \sum_{\substack{\pi \in NC(m)\\ \pi \vee \wh{0}_m = 1_m}} \kappa_\pi^{\mc P}[a_1 \wedge_k a_2 \wedge_k \dotsb \wedge_k a_m].
\end{equation*}
Now since $i_j$ are not all equal, the condition $\pi \vee \wh{0}_m = 1_m$ implies that $\kappa_\pi^{\mc P}$ will connect two elements from different algebras and will therefore be zero by assumption.  So the mixed cumulants vanish, which is Speicher's condition for freeness.
\end{proof}

\subsection{ Examples of \texorpdfstring{$\mc P$-traces}{P-traces}}  
We will now give some examples of planar algebra traces, which are defined combinatorially using the free cumulants of a given compactly supported probability measure $\nu$ on $\R$.  
In the case that $\nu$ is the free Poisson distribution we recover the \textit{Voiculescu trace} on $\mc P$ from \cite{gjs1}.  
In the rest of the paper we will be especially interested in the trace obtained by taking $\nu$ to be the semicircle law, which we will call the \textit{2-cabled Voiculescu trace}.

First we require some combinatorial preliminaries.  Given a non-crossing partitition $\pi \in NC(k)$, define its ``fattening'' $\widetilde \pi \in TL(k)$ as follows:  For each block $V = (i_1,\dotsc,i_s)$ of $\pi$, we add to $\widetilde \pi$ the pairings $(2i_1-1,2i_s), (2i_1,2i_2-1),\dotsc$, $(2i_{s-1},2i_s-1)$.  It is not hard to see that
\begin{equation*}
 \widetilde \pi \vee \cup \cup \dotsb \cup = \widehat \pi,
\end{equation*}
where $\widehat \pi \in NC_2(2k)$ is obtained by partitioning the pairs $(1,2),(3,4),\dotsc,(2k-1,2k)$ according to $\pi$.  

\begin{example} The fattening of $\pi = \{\{1,4,5\},\{2,3\},\{6\}\} \in NC_6$ is given by
%
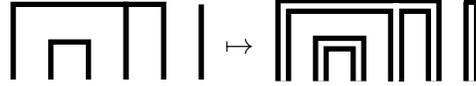
\begin{equation*}
\begin{array}{ccc}
 \begin{tikzpicture}[scale=.5]
   \draw[line width=2](0,0) -- ++(0,2) -- ++(3,0) -- ++(0,-2) --++(0,2) --++(1,0)--++(0,-2);
   \draw[line width=2](1,0) -- ++(0,1) --++(1,0) --++(0,-1);
   \draw[line width=2](5,0)--++(0,2);
  \end{tikzpicture}\end{array} \mapsto \begin{array}{c}  \begin{tikzpicture}[scale=.5]
   \draw[line width=2,double,double distance=2pt](0,0) -- ++(0,2) -- ++(3,0) -- ++(0,-2) --++(0,2) --++(1,0)--++(0,-2);
   \draw[line width=2,double,double distance=2pt](1,0) -- ++(0,1) --++(1,0) --++(0,-1);
      \draw[line width=2,double distance=2pt](5,0)--++(0,2.1);
      \draw[line width=2](4.8,2.1)--(5.2,2.1);
  \end{tikzpicture}\end{array}
\end{equation*} 

\end{example}

The following relationship between rotation in $TL(k)$ and the Kreweras complement $Kr$ on $NC(k)$ was proved in \cite{cs1}.

\begin{proposition} \label{kreweras}
If $\pi \in NC(k)$ then
\begin{equation*}
 \widetilde{K(\pi)} = \rho(\widetilde{\pi}),
\end{equation*}
where $\rho$ is the counter-clockwise rotation on $TL(k)$. 
\end{proposition}

%

Now let $\nu$ be a compactly supported distribution on $\mathbb R$, and let $\mathbb \kappa_n$ be its free cumulants.  Recall that we have
\begin{equation*}
 \nu(x^n) = \sum_{\pi \in NC(n)} \kappa_\pi,
\end{equation*}
where
\begin{equation*}
 \kappa_\pi = \prod_{V \in \pi} \kappa_{|V|}.
\end{equation*}
The free convolution power $\nu^{\boxplus t}$ is determined by $\kappa_n^{\nu^{\boxplus t}} = t \cdot \kappa_n^{\nu}$.  In terms of moments this is defined for all $t > 0$, and it is known that $\nu^{\boxplus t}$ is a compactly supported measure on $\mb R$ for $t \geq 1$.  However for $t > 1$ it is not always true that $\nu^{\boxplus (1/t)}$ corresponds to a measure, when it does we say that $\nu$ is \textit{$t$-times freely divisible}.  $\nu$ is called \textit{infinitely freely divisible} if it is $t$-times freely divisible for all $t > 1$.

Now define $T_m \in TL(m)$ by
\begin{equation*}
 T_m = \sum_{\pi \in NC(m)} \kappa_\pi \cdot \widetilde \pi,
\end{equation*}
or equivalently in terms of cumulants:
\begin{equation*}
 \begin{tikzpicture}[scale=.5]
\draw[thick] (.75,1) arc(180:0:.5cm);
\draw[thick] (2.25,1) arc(180:0:.5cm);
\draw[thick] (5.25,1) arc(0:180:.5cm);
\draw[thick] (.25,1) arc(180:90:.75cm) -- (5,1.75) arc(90:0:.75cm);
\node[scale=.75] at (3.75,1.25) {$\dotsb$};

\node[left] at (-.25,1.375) {$\kappa_m^\mc P = \kappa_m(\nu)\; \cdot$};
\end{tikzpicture}
\end{equation*}
Let $\tau^\nu = (T_m)_{m \geq 0}$ be the associated $\mc P$-functional.

\begin{example}\hfill
\begin{enumerate}
 \item If $\nu$ is the free Poisson(1) distribution then $\kappa_\pi = 1$ for all $\pi$, and so $T_m = \sum TL(m)$ and $\tau^\nu$ is the \textit{Voiculescu trace} of \cite{gjs1}.
 \item If $\nu$ is the semicircle law $S(0,1)$, then $\kappa_\pi$ is $1$ or $0$ depending on whether $\pi \in NC_2(m) \sim TL(m/2)$.  So $T_m$ is equal to the sum over all doubled $TL(m/2)$ diagrams, and in particular is zero if $m$ is odd.  We will refer to $\tau = \tau^\nu$ as the \textit{2-cabled Voiculescu trace}.  
\end{enumerate}
\end{example}

\begin{remark}\label{2cable-remark}
The two examples above are closely related.  Let $\mc P$ be a planar algebra and let $\mc P^c$ be its 2-cabling \cite{palg}.  Then the 2-cabled Voiculescu trace on $Gr_{2k}(\mc P)$ restricts to the standard Voiculescu trace on $Gr_k(\mc P^c)$ under the obvious inclusion $Gr_k(\mc P^c) \subset Gr_{2k}(\mc P)$.  This extends to an index 2 inclusion of the associated von Neumann algebras.  

Note also that if one takes the unshaded 2-cabling $\mc P^{uc}$, then there is a natural trace preserving isomorphism between the 2-cabled Voiculescu trace on $Gr_{2k}(\mc P)$ and the standard (unshaded) Voiculescu trace on $Gr_k(\mc P^{uc})$.  The unshaded Voiculescu trace has been considered by Brothier in \cite{bro1}.
\end{remark}

\begin{proposition} \hfill
\begin{itemize}
\item The distribution of $\cup$ with respect to $\tau^\nu$ is $\nu^{\boxplus \delta}$.
\item The distribution of $\delta^{-1}\cdot \corner \in Gr_1(\mc P)$ with respect to $\tau^\nu$ is $\nu^{\boxplus (1/\delta)}$.
\end{itemize}
\end{proposition}

\begin{proof}
By Proposition \ref{product_cumulants} we have
\begin{equation*}
 \begin{tikzpicture}[scale=.5]
  \draw[thick] (.75,1) arc(180:0:.5cm);
\draw[thick] (2.25,1) arc(180:0:.5cm);
\draw[thick] (5.25,1) arc(0:180:.5cm);
\draw[thick] (.25,1) arc(180:90:.75cm) -- (5,1.75) arc(90:0:.75cm);
\node[scale=.75] at (3.75,1.25) {$\dotsb$};
\draw[thick] (.25,1) arc(180:360:.25);
\draw[thick] (1.75,1) arc(180:360:.25);
\draw[thick] (5.25,1) arc(180:360:.25);
\node[left] at (-.25,1.25) {$\kappa_m[\cup,\dotsc,\cup] = \kappa_m^\mc P[\cup \wedge_0 \dotsb \wedge_0 \cup] = \kappa_m(\nu) \; \cdot $};
\node[right] at (6,1.25) {$ = \delta \cdot \kappa_m(\nu),$};
 \end{tikzpicture}
\end{equation*}
which shows that the distribution of $\cup$ is $\nu^{\boxplus \delta}$.  Likewise we have
\begin{equation*}
\begin{tikzpicture}[scale=.5]
  \draw[thick] (.75,1) arc(180:0:.5cm);
\draw[thick] (2.25,1) arc(180:0:.5cm);
\draw[thick] (5.25,1) arc(0:180:.5cm);
\draw[thick] (.25,1) arc(180:90:.75cm) -- (5,1.75) arc(90:0:.75cm);
\node[scale=.75] at (3.75,1.25) {$\dotsb$};
\draw[thick] (.25,1) arc(0:-90:.25) arc(90:270:.25) -- (6,.25) arc(-90:90:.25) arc(270:180:.25);
\draw[thick] (.75,1) arc(180:360:.5);
\draw[thick] (2.25,1) arc(180:360:.5);
\draw[thick] (4.25,1) arc(180:360:.5);
\node[left] at (-.35,1) {$ \kappa_m[\delta^{-1}\corner,\dotsc,\delta^{-1}\corner] = \delta^{-1} \delta^{-m} \kappa_m(\nu) \; \cdot $};
\node[right] at (6.25,1) {$ = \delta^{-1} \cdot \kappa_m(\nu),$};
\end{tikzpicture}
\end{equation*}
so that $\delta^{-1} \cdot \corner$ has distribution $\nu^{\boxplus (1/\delta)}$.
\end{proof}

It follows from the result above that for $\tau^\nu$ to extend to a positive $\mc P$-trace it is necessary that $\nu$ is $\delta$-times freely divisible.  We will now show that this condition is also sufficient.  First we need a lemma.

\begin{lemma}\label{proj-expectations}
Let $p_1,\dotsc,p_m$ be mutually orthogonal projections in a noncommutative probability space $(A,\phi)$.  Then for $\pi \in NC(n)$ we have
\begin{equation*}
 \phi_{Kr(\pi)}[p_{i_1},\dotsc,p_{i_n}] = \begin{cases} \phi(p_{i_n}) \displaystyle\prod_{V = (l_1 < \dotsb < l_s) \in \pi} \phi(p_{i_{l_1}}) \dotsb \phi(p_{i_{l_{s-1}}}), & Kr(\pi) \leq \ker \bf i\\ 0, & \text{otherwise}\end{cases}.
\end{equation*}

\end{lemma}

\begin{proof}
 The result is clear when $\pi = 1_n$ is the partition with only one block.  Otherwise let $V = (l+1,\dotsc,l+s)$ be an interval of $\pi$ with $l > 0$.  Note that $Kr(\pi)$ is obtained by taking $Kr(\pi \setminus V)$, adding singletons $l+1,\dotsc,l+s-1$ and adding $l+s$ to the block containing $l$.  We therefore have
 \begin{equation*}
  \phi_{Kr(\pi)}[p_{i_1},\dotsc,p_{i_n}] = \delta_{i_{l+s} = i_l} \cdot \phi_{Kr(\pi \setminus V)}[p_{i_1},\dotsc,p_{i_l},p_{i_{l+s+1}},\dotsc,p_n] \cdot \phi(p_{i_{l+1}})\dotsb \phi(p_{i_{l+s-1}}),
 \end{equation*}
and the result follows by induction on the number of blocks of $\pi$.
\end{proof}

\begin{theorem}\label{div-trace}
The law of $\nu$ is $\delta$-times freely divisible if and only if $\tau^\nu$ is a positive $\mc P$-trace.
\end{theorem}

\begin{proof}
We have already shown that if $\tau^\nu$ is positive, it must be that $\nu$ is $\delta$-times freely infinitely divisible.

Let $\Gamma$ be the principal graph of $\mc P$.  We must construct a von Neumann algebra $M$ with the following properties:
\begin{itemize}
\item There is an inclusion $l^\infty(\Gamma_+) \hookrightarrow M$ and a positive conditional expectation $E:M \to l^\infty(\Gamma_+)$.
\item There are operators $\{X_{e,f^o}:e,f \in E_+(\Gamma), \; t(e) = t(f)\}$ satisfying $X_{e,f^o}^* = X_{f,e^o}$ and $p_vX_{e,f^o}p_w = \delta_{v = s(e)} \delta_{w = s(f)} X_{e,f^o}$.
\item The $l^\infty(\Gamma_+)$-valued cumulants of $X_{e,f^o}$ are given by:
\begin{equation*} \tag{$*$} \label{graph-cumulants}
\kappa_n^{l^\infty(\Gamma)}[X_{e_n,f_1^o}p_{v_1},\dotsc, X_{e_{n-1},f_n^o}p_{v_n}] = \begin{cases} \kappa_n(\nu)\frac{\mu_{v_1}\dotsb \mu_{v_{n-1}}}{\mu_{t(e_n)}^{n-1}} \cdot p_{v_n}, & e_i = f_i,\; s(e_i) = v_i, \; 1 \leq i \leq n\\ 0, & \text{otherwise}\end{cases}.
\end{equation*}
\end{itemize}
First observe that the algebras $M_w$ generated by $\{X_{e,f^o}: e,f^o \in E_+(\Gamma),\; t(e) = t(f) = w\}$, for $w \in \Gamma_-$, are free with amalgamation over $l^\infty(\Gamma_+)$ by \eqref{graph-cumulants} and the characterization of freeness as vanishing of mixed cumulants.  This allows us to reduce to the case that $\Gamma_-$ consists of a single vertex $w$, as we can then recover the general result by forming an amalgamated free product.  Since $\Gamma$ is assumed locally finite, in this case $\Gamma$ must be finite.  

Now let $N = |E_+(\Gamma)|$ and consider the algebra $M_N(\C\Gamma_+)$, with matrix units $(V_{ef})_{e,f \in E_+(\Gamma)}$.  Let $\mc D$ denote the diagonal subalgebra, $\mc D \simeq \C\Gamma_+ \otimes \C E_+$.  Now if $X$ is a $\mc D$-valued random variable which is free from $M_N(\C\Gamma_+)$ with amalgamation over $\mc D$, and if we set $X_{e,f^o} = V_{1e}XV_{f1}$ (where $1$ is any fixed edge), then by \cite[Lemma 3.5]{csp2} we have:
\begin{equation*}
\kappa_n^{\C\Gamma_+}\bigl[X_{e_n,f_1^o}p_{v_1},X_{e_1,f_2^o}p_{v_2},\dotsc,X_{e_{n-1},f_n^o}p_{v_n}\bigr]\cdot V_{e_ke_k} = \begin{cases}\kappa_{n}^{\mc D}\bigl[Xp_{v_1}V_{e_1e_1},\dotsc,Xp_{v_n}V_{e_ne_n}\bigr], & e_i = f_i\\ 0, & \text{otherwise}\end{cases}
\end{equation*}
So it suffices to find an operator $X$ in a $\mc D$-valued W$^*$-probability space which satisfy
\begin{equation*} \tag{$\dagger$} \label{diag-cumulants}
\kappa_n^{\mc D}\bigl[Xp_{v_1}V_{e_1e_1},\dotsc,Xp_{v_k}V_{e_ke_k}\bigr] = \begin{cases} \biggl(\kappa_n(\nu)\cdot \frac{\mu_{v_1}\dotsb \mu_{v_{k-1}}}{\mu(w)^{k-1}}\biggr) \cdot p_{v_n}V_{e_ne_n}, & v_i = s(e_i)\\0, & \text{otherwise}\end{cases}.
\end{equation*}

Let $V_e = p_{s(e)}V_{ee}$ and let $\mc B \subset \mc D$ denote the span of $\{V_e: e \in E_+\}$, so that $\mc B \simeq \C E_+(\Gamma)$.  Define a tracial state $\phi$ on $\mc B$ by 
\begin{equation*}
 \phi(V_e) = \frac{\mu_{s(e)}}{\delta \mu_w}.
\end{equation*}
Note that
\begin{equation*}
 \phi(1_{\mc B}) = \sum_{e \in E_+} \phi(V_e) = \frac{V_{e_n}}{\delta \mu_w}\sum_{e \in E_+} \mu_{s(e)} = 1
\end{equation*}
by the eigenvector condition for $\mu$.  Now let $Y$ be free from $\mc B$ with respect to $\phi$ and have distribution $\nu^{\boxplus \delta^{-1}}$, and set $X = \delta\cdot Y$.  We then have
\begin{align*}
 E_{\mc B}[XV_{e_1}\dotsb XV_{e_n}] &= \frac{V_{e_n}}{\phi(V_{e_n})}\phi\bigl(XV_{e_1}\dotsb XV_{e_n}\bigr)\\
 &= \frac{V_{e_n}}{\phi(V_{e_n})}\sum_{\pi \in NC(n)} \delta^{n}\kappa_\pi[Y,\dotsc,Y] \phi_{Kr(\pi)}[V_{e_1},\dotsc,V_{e_n}]
\end{align*}
where $Kr$ denotes the Kreweras complement (see \cite{nsp}).  By Lemma \ref{proj-expectations} above we may continue with
\begin{equation*}
 E_{\mc B}[XV_{e_1}\dotsb XV_{e_n}] = V_{e_n} \cdot \sum_{\substack{\pi \in NC(n)\\ Kr(\pi) \leq \ker \bf e}} \delta^{n-|\pi|}\kappa_\pi(\nu) \prod_{V = (l_1 < \dotsb < l_s) \in \pi} \frac{\mu_{s(e_{l_1})}\dotsb \mu_{s(e_{l_{s-1}})}}{\delta^{s-1}\mu_w^{s-1}}
\end{equation*}
Since cumulants are uniquely determined by the moment-cumulant formula, it follows that
\begin{equation*}
 \kappa_{n}^{\mc B}[XV_{e_1}\dotsb XV_{e_n}] = \kappa_n(\nu) \frac{\mu_{s(e_1)}\dotsb \mu_{s(e_{n-1})}}{\mu_w^{n-1}}\cdot V_{e_n}.
\end{equation*}
Note that the condition $Kr(\pi) \leq \ker \bf e$ in the formula above is forced by the fact that the $V_e$ are mutually orthogonal projections.

Finally, we extend $X$ to a $\mc D$-valued random variable by taking the direct sum of the von Neumann algebra generated by $X$ and $\mc B$ with $\mc D \ominus \mc B$.  We then have that the $\mc D$-valued cumulants of $X$ satisfy \eqref{diag-cumulants}, which completes the proof.

\end{proof}

\subsection{ \texorpdfstring{$\mc P$-distributions}{P-distributions}}
\label{sec:defOfEv}
Now suppose that $\tau$ is a faithful $\mc P$-trace, and let $Y$ be a self-adjoint element of $Gr_1(\mc P)$.  There is a homomorphism $\ev_{Y}:Gr_0 \to Gr_0$ defined by
\begin{equation*}
  \begin{tikzpicture}
   \draw[Box] (-.125,.5) rectangle (4.125,1); \node at (2,.75) {$P$}; \node[marked,scale=.8,above left] at (-.125, 1) {};
   \draw[Box] (.25,1.25) rectangle (.75,1.75); \node at (.5,1.5) {$Y$}; \node[marked,scale=.8, above left] at (.25,1.75) {};
   \draw[Box] (1.25,1.25) rectangle (1.75,1.75); \node at (1.5,1.5) {$Y$}; \node[marked,scale=.8, above left] at (1.25,1.75) {};
   \draw[Box] (3.25,1.25) rectangle (3.75,1.75); \node at (3.5,1.5) {$Y$}; \node[marked,scale=.8, above left] at (3.25,1.75) {};
   \node at (2.5,1.5) {$\dotsb$};
   \foreach \x in {.25,1.25,3.25} {\draw[thick] (\x,1.5) arc(90:180:.125) -- ++(0,-.375);}
   \foreach \x in {.75,1.75,3.75} {\draw[thick] (\x,1.5) arc(90:0:.125) -- ++(0,-.375);}
      \foreach \x in {.5,1.5,3.5} {\draw[thickline] (\x,1.75) -- ++(0,.175);}
   \node[left] at (-.5,1) {$\ev_Y(P) = $};
  \end{tikzpicture}
\end{equation*}
Define the \textit{$\mc P$-distribution} $\tau^{(Y)}$ of $Y$ by $\tau^{(Y)}(a) = \tau(\mathrm{ev}_Y(a))$ for $a \in Gr_0(\mc P)$.  Diagrammatically, we have $\tau^{(Y)} = (T_k^{(Y)})_{k \geq 1}$ with
\begin{equation*}
 \begin{tikzpicture}
   \draw[Box] (.25,1.25) rectangle (.75,1.75); \node[scale=.8] at (.5,1.5) {$Y_{n_1}$}; \node[marked,scale=.8, above left] at (.25,1.75) {};
   \draw[Box] (1.25,1.25) rectangle (1.75,1.75); \node[scale=.8] at (1.5,1.5) {$Y_{n_2}$};\node[marked,scale=.8, above left] at (1.25,1.75) {};
   \draw[Box] (3.25,1.25) rectangle (3.75,1.75); \node[scale=.8] at (3.5,1.5) {$Y_{n_k}$};\node[marked,scale=.8, above left] at (3.25,1.75) {};
   \node at (2.5,1.5) {$\dotsb$};
   \foreach \x in {.25,1.25,3.25} {\draw[thick] (\x,1.5) arc(90:180:.125) -- ++(0,-.375);}
   \foreach \x in {.75,1.75,3.75} {\draw[thick] (\x,1.5) arc(90:0:.125) -- ++(0,-.375);}
      \foreach \x in {.5,1.5,3.5} {\draw[thickline] (\x,1.75) -- ++(0,.25);}
   \draw[Box] (0,2) rectangle (4,2.5); \node at (2,2.25) {$T_n$}; \node[marked,scale=.8, below right] at (4,2) {};
   \node[left] at (-.5,1.625) {$T_k^{(Y)} = \displaystyle \sum_n \sum_{n_1 + \dotsb + n_k =n}$};
 \end{tikzpicture}
\end{equation*}
where $Y= \sum Y_n$ with $Y_n \in P_{n,1}$ and $Y_n = 0$ for all but finitely many $n$.  

In the case of a graph planar algebra $\mc P^\Gamma$ this corresponds to the standard evaluation of polynomials:
\begin{lemma}
Suppose $Y \in Gr_1(\mc P^\Gamma)$.  Under the identifications of Proposition \ref{graph-ids}, $\ev_Y$ is the restriction to $Gr_0(\mc P^\Gamma)$ of the homomorphism $\ms A_\Gamma \to \ms A_\Gamma$ mapping $X_{e,f^o}$ to $Y_{e,f^o}$.  
\end{lemma}


\section{Free analysis and planar algebras} \label{sec:freeAnalysis}

Let $\mc P$ be a planar algebra.  Let $\ms N$ denote the operator on $Gr_0(\mc P)$ which multiplies $x \in P_n$ by $n$.  Let $(Gr_0(\mc P))_0$ denote the subalgebra spanned by $\{P_n:n > 0\}$, and let $\Pi$ denote the projection onto $(Gr_0(\mc P))_0$ which sends $x \in P_n$ to $x\cdot \delta_{n > 0}$.  Define $\Sigma$ to be the composition of $\Pi$ with the inverse of $\ms N$.  

Now define the \textit{free difference quotient} $\partial:Gr_1 \to Gr_1 \boxtimes Gr_1^{op}$ to be the derivation defined by
\begin{equation*}
\begin{tikzpicture}[scale=.5]
 \draw[Box] (0,0) rectangle (2,1); \node at (1,.5) {$x$};\node[marked,scale=.8,above left] at (0,1) {};
 \draw[thick] (0,.5) --++(-.5,0);
 \draw[thick] (2,.5) arc(90:-90:.3 and .375) --++(-2.5,0);
 \draw[thickline] (1.5,1) arc(180:0:.7cm and .5cm) -- node[rcount,scale=.75] {$2k$} ++(0,-2);
 \draw[thickline] (.5,1) --++(0,1.5);
 \draw[thick] (1.15,1) arc(180:0:1.2cm and .85cm) -- ++(0,-.75) arc(180:270:.5cm) --++(.3,0);
 \draw[thick] (.85,1) arc(180:0:1.5cm and 1.05cm) -- ++(0,-.25) arc(180:270:.5cm);
 
 \node[left] at (-1,.5) {$x \mapsto \displaystyle\sum_{k = 0}^{n-1} $};
\end{tikzpicture}
\end{equation*}
for $x \in P_{n,1} = P_{n+1}$.  Note that if $X = \tikz{\draw (0,0) arc(-90:0:.15 and .25); \draw (.5,0) arc(270:180:.15 and .25);}$ and $B = P_{0,1}$, then the restriction to $B[X]$ is the usual free difference quotient $\partial_{X:B}$.

The \textit{conjugate variable} $\xi  \in L^2(Gr_1,\tau_1)$ is defined by the requirement
\begin{equation*}
 \tau_1(\xi \wedge_1 x) = (\tau_1 \boxtimes \tau_1^{op}) \bigl(\partial(x)\bigr), \qquad \forall x \in Gr_1.
\end{equation*}
Diagrammatically, the condition is
\begin{equation*}
\begin{tikzpicture}[yscale=.75]
\draw[Box] (0,0) rectangle (1,1); \node at (.5,.5) {$\xi$}; \node[marked,scale=.8,above left] at (0,1) {};
\draw[Box] (1.5,0) rectangle (2.5,1); \node at (2,.5) {$x$}; \node[marked,scale=.8,above left] at (1.5,1) {};
\draw[Box] (.25,1.25) rectangle (2.25,2); \node at (1.25,1.625) {$\tau$}; \node[marked,scale=.8,below right] at (2.25,1.25) {};

\draw[thick] (1,.5) -- (1.5,.5);
\draw[thick] (0,.5) arc(90:270:.25 and .375) --++(2.5,0) arc(-90:90:.25 and .375);

\foreach \x in {.5,2} {\draw[thickline] (\x,1) --++(0,.25);}
\begin{scope}[xshift=7cm]
\draw[Box] (0,0) rectangle (3.25,1); \node at (1.625,.5) {$x$}; \node[marked,scale=.8,above left] at (0,1) {};
\draw[Box] (.25,1.25) rectangle (1.25,2); \node at (.75,1.625) {$T_m$}; \node[marked,scale=.8,below right] at (1.25,1.25) {};
\draw[Box] (2,1.25) rectangle (3,2); \node at (2.5,1.625) {$T_k$}; \node[marked,scale=.8, below right] at (3,1.25) {};

\draw[thick] (0,.5) arc(270:180:.5) --++(0,1) arc(180:90:.5) --++(1,0) arc(90:0:.5) --++(0,-1);
\draw[thick] (3.25,.5) arc(-90:0:.5) --++(0,1) arc(0:90:.5) --++(-1,0) arc(90:180:.5) --++(0,-1);
\draw[thickline] (.75,1) --++(0,.25);
\draw[thickline] (2.5,1) --++(0,.25);

\node[left] at (-.75,1) {$= \displaystyle \sum_{k + m = n-1} \delta^{-1} \; \cdot$};
\end{scope}
\end{tikzpicture}
\end{equation*}
for $x \in P_n$. 

The \textit{cyclic gradient} $\ms D:Gr_0 \to Gr_1$ is defined by 
\begin{equation*}
\begin{tikzpicture}[scale=.5]
 \draw[Box] (0,0) rectangle (2,1); \node at (1,.5) {$x$}; \node[marked,scale=.8,above left] at (0,1) {};
 \draw[thickline] (.5,1) --++(0,1.3);
 \draw[thickline] (1.5,1) arc(180:0:.6cm and .5cm) -- ++(0,-1) arc(0:-90:.5cm) -- node[rcount,scale=.75] {$2k$} ++(-2.2,0) arc(270:180:.5cm) --++(0,2.3);
 \draw[thick] (1.15,1) arc(180:0:1cm and .85cm) --++(0,-1.25) arc(0:-90:.75cm) -- ++(-3.75,0);
 \draw[thick] (.85,1) arc(180:0:1.3cm and 1.05cm) --++(0,-1.25) arc(180:270:.75cm) --++(.5,0);
 
 \node[left] at (-2,.5) {$x \mapsto \displaystyle \sum_{k=0}^{n-1}$};
\end{tikzpicture}
\end{equation*}

We define the operator $\#:Gr_1 \boxtimes Gr_1^{op} \times Gr_1 \to Gr_1$ by
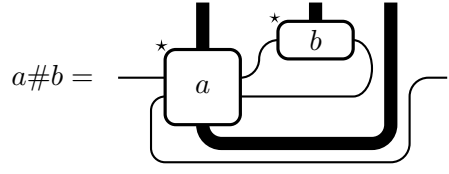
\begin{equation*}
\begin{tikzpicture}[scale=.5]
\draw[Box] (0,0) rectangle (2,2); \node at (1,1) {$a$}; \node[marked,scale=.9, above left] at (0,2) {};
\draw[Box] (3,1.75) rectangle (5,2.75); \node at (4,2.25) {$b$}; \node[marked,scale=.8, above left] at (3,2.75) {};

\draw[thickline] (1,2) -- ++(0,1.25);
\draw[thickline] (4,2.75) --++(0,.5);
\draw[thickline] (1,0) arc(180:270:.5) -- ++(4,0) arc(-90:0:.5) --++(0,3.25);
\draw[thick] (2,1.25) arc(-90:0:.5) arc(180:90:.5);
\draw[thick] (2,.75) --++(3,0) arc(-90:90:.5 and .75);
\draw[thick] (0,1.25) --++(-1.25,0);
\draw[thick] (0,.75) arc(90:180:.375) --++(0,-1) arc(180:270:.375) --++(6,0) arc(-90:0:.5) --++(0,1.25) arc(180:90:.5) --++(.5,0);

\node[left] at (-1.625,1.25) {$a \# b = $};
\end{tikzpicture}
\end{equation*}
We  use $\cdot$ to denote the operator $\cdot:Gr_1(\mc P) \times Gr_1(\mc P) \to Gr_0(\mc P)$ defined by
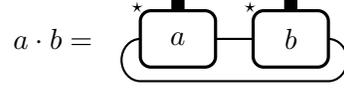
\begin{equation*}
\begin{tikzpicture}[yscale=.75]
\draw[Box] (0,0) rectangle (1,1); \node at (.5,.5) {$a$}; \node[marked,scale=.8,above left] at (0,1) {};
\draw[Box] (1.5,0) rectangle (2.5,1); \node at (2,.5) {$b$}; \node[marked,scale=.8,above left] at (1.5,1) {};

\draw[thick] (1,.5) -- (1.5,.5);
\draw[thick] (0,.5) arc(90:270:.25 and .375) --++(2.5,0) arc(-90:90:.25 and .375);

\foreach \x in {.5,2} {\draw[thickline] (\x,1) --++(0,.25);}
\node[left] at (-.5,.5) {$a \cdot b = $};
\end{tikzpicture}
\end{equation*}

We define the cyclic symmetrizer $\ms S:(Gr_0(\mc P))_0 \to (Gr_0(\mc P))_0$ by
\begin{equation*}
\begin{tikzpicture}[scale=.5]
 \draw[Box] (0,0) rectangle (2,1); \node at (1,.5) {$x$}; \node[marked,scale=.8,above left] at (0,1) {};
 \draw[thickline] (.625,1) --++(0,1.3);
 \draw[thickline] (1.375,1) arc(180:0:.6cm and .5cm) -- ++(0,-1) arc(0:-90:.5cm) -- node[rcount,scale=.75] {$2k$} ++(-2.2,0) arc(270:180:.5cm) --++(0,2.3);

 \node[left] at (-2,.5) {$\ms S(x) = \displaystyle \frac{1}{n} \sum_{k=0}^{n-1}$};
\end{tikzpicture}
\end{equation*}
for $x \in P_n$. 

\begin{definition}
We define the free Fisher information $\Phi_{\mc P}^*(\tau)\in [0,+\infty]$ of $\tau$ to be $+\infty$ if the conjugate variable $\xi$ does not exist, and otherwise by the formula
$$
\begin{tikzpicture}[yscale=.75]
\draw[Box] (0,0) rectangle (1,1); \node at (.5,.5) {$\xi$}; \node[marked,scale=.8,above left] at (0,1) {};
\draw[Box] (1.5,0) rectangle (2.5,1); \node at (2,.5) {$\xi$}; \node[marked,scale=.8,above left] at (1.5,1) {};
\draw[Box] (.25,1.25) rectangle (2.25,2); \node at (1.25,1.625) {$\tau$}; \node[marked,scale=.8,below right] at (2.25,1.25) {};

\draw[thick] (1,.5) -- (1.5,.5);
\draw[thick] (0,.5) arc(90:270:.25 and .375) --++(2.5,0) arc(-90:90:.25 and .375);

\foreach \x in {.5,2} {\draw[thickline] (\x,1) --++(0,.25);}
\node[left] at (-0.5,.5) {$\Phi_{\mc P}^*(\tau)^2 = \tau_1 (\xi \wedge_1\xi) =\delta^{-1}\!\!\!\!$};
\end{tikzpicture}
$$
\end{definition}

\begin{proposition}\label{prop:restriction}
Let $P\subset P'$ be planar algberas, and assume that $\tau'$ is a $P'$-trace.  Let $\tau$ denote the restriction of $\tau'$ to a $P$-trace.  Then $\Phi^*(\tau)\leq \Phi^*(\tau')$.  In particular, if $\Phi^*(\tau')<\infty$, so is $\Phi^*(\tau)$. 
\end{proposition}
\begin{proof}
If $\Phi^*(\tau')=+\infty$ there is nothing to prove, so we may assume that the conjugate variable $\xi'$ to $\tau'$ exists.  
Note that the restriction of $\partial$ to $Gr_1(P)$ is the free difference quotient associated to $P$.  It follows that if we denote by $\xi$ the orthogonal projection of $\xi$ onto $L^2(Gr_1(P),\tau)\subset L^2(Gr_1(P'),\tau')$, then $\xi$ is the conjugate variable for $\tau$. Since orthogonal projections are contractive, the inequality follows.
\end{proof}

\subsection{Free analysis on graph planar algebras.}

\subsubsection{Path algebra construction.}
Let $\Gamma = \Gamma_+ \cup \Gamma_-$ be a locally finite bipartite graph, and let $\ms A_\Gamma$ be the even part of the path algebra (see Section \ref{sec:background}).  For pairs of edges $(e,f^o)$ in $E_+$ with $t(e) = t(f)$, define $\partial_{(e,f^o)}:\ms A_\Gamma \to \ms A_\Gamma \otimes (\ms A_\Gamma)^{op}$ to be the derivation sending $X_{e,f^o}$ to $p_{s(e)} \otimes p_{s(f)}$ and sending all other generators to zero.  Explicitly we have
\begin{equation*}
 \partial_{(e,f^o)}[X_{e_1,f_1^o}\dotsb X_{e_m,f_m^o}] = \sum_{(e_j,f_j^o) = (e,f^o)} X_{e_1,f_1^o}\dotsb X_{e_{j-1},f_{j-1}^o}p_{s(e)} \otimes p_{s(f)}X_{e_{j+1},f_{j+1}^o}\dotsb X_{e_m,f_m^o}.
\end{equation*}

Define $\ms D_{(e,f^o)}: \ms A_\Gamma \to \ms A_\Gamma$ by 
\begin{equation*}
 \ms D_{(e,f^o)}[X_{e_1,f_1^o}\dotsb X_{e_m,f_m^o}] = \mu_{s(e_1)} \cdot \sum_{(e_j,f_j^o) = (f,e^o)} p_{s(e)}X_{e_{j+1},f_{j+1}^o}\dotsb X_{e_m,f_m^o} X_{e_1,f_1^o}\dotsb X_{e_{j-1},f_{j-1}^o}p_{s(f)}
\end{equation*}
Note that $\ms D_{(e,f^o)}$ is zero on the complement of the span of ``loops'' $X_{e_1,f_1^o}\dotsb X_{e_m,f_m^o}$ with $s(e_1) = s(f_m)$, which is identified with $Gr_0(\mc P^\Gamma)$ by Proposition \ref{graph-ids}.  Observe also that we have the relation $\ms D_{(e,f^o)}(Y) = m(\sigma(\partial_{(f,e^o)}[\sum_v \mu_v \cdot Y p_v]))$ where $m:\ms A_\Gamma \otimes \ms A_{\Gamma}^{op} \to \ms A_\Gamma$ is multiplication and $\sigma(a \otimes b) = b\otimes a$.  Define the cyclic gradient $\ms D(G) = (\ms D_{e,f^o}(G))_{(e,f^o)} \in (\ms A_\Gamma)^N$ where $N = \#\{(e,f^o):t(e) = t(f)\}$.

Define the free Jacobian $\ms J:(\ms A_\Gamma)^N \to M_N(\ms A_\Gamma \otimes \ms A_\Gamma^{op})$ by $(\ms J Y)_{(e,f^o),(g,h^o)} = \partial_{(g,h^o)}(Y_{e,f^o})$.

For $q = \sum a_i \otimes b_i \in \ms A_\Gamma \otimes \ms A_\Gamma^{op}$ and $g \in \ms A$ define 
\begin{equation*}
 q \# g = \sum_i a_i g b_i.
\end{equation*}
For $q = (q_{ij}) \in M_{N}(\ms A_\Gamma \otimes \ms A_\Gamma^{op})$, and $g = (g_j)$, $h = (h_j) \in \ms A_\Gamma^N$ define
\begin{align*}
 (q \# g)_j &= \sum_i q_{ji} \# g_i\\
 g \cdot h &= \sum_i g_i h_i
\end{align*}

\begin{proposition}\label{graph-derivatives}
With the identifications of Proposition \ref{graph-ids}, $\ms D$ maps $\ms A_\Gamma$ into $Gr_1(\mc P^\Gamma)$ and the restriction of $\ms J$ maps $Gr_1(\mc P^\Gamma)$ into $Gr_1(\mc P^\Gamma) \boxtimes Gr_1(\mc P^\Gamma)^{op}$.  The restrictions of $\#$ and $\cdot$ map $Gr_1(\mc P^\Gamma) \boxtimes Gr_1(\mc P^\Gamma)^{op} \times Gr_1(\mc P^\Gamma) \to Gr_1(\mc P^\Gamma)$ and $Gr_1(\mc P^\Gamma) \times Gr_1(\mc P^\Gamma) \to Gr_1(\mc P^\Gamma)$. Moreover, these restrictions agree with the diagrammatic formulas given above (where $\ms J$ corresponds to $\partial$).
\end{proposition}

\begin{proof}
These identifications follow easily from the definition of $\mc P^\Gamma$.
\end{proof}

\subsubsection{Non-tracial construction.}\label{sec:embedding-nt}
Consider the algebra $\ms B_\Gamma$ generated by indeterminates $Y_{e,f^o}$ associated to pairs $(e,f)\in E_+$ satisfying $t(e)=t(f)$.  Define the $*$-structure on this algebra by setting $Y_{e,f^o}^* = Y_{f,e^o}$.  Finally, let $\phi$ be a free quasi-free state on this algebra; in other words, $\phi$ is a linear functional for which only second order cumulants are nonzero.  To specify $\phi$, it is sufficient to specify these second order cumulants, or, equivalently, the covariances:
$$
\phi(Y_{e_1,f_1^o} Y_{e_2,f_2^o}^*) = \delta_{e_1=e_2}\delta_{f_1=f_2} \mu_{s(f_1)}\mu_{s(e_1)}^2.
$$
Note that $\phi$ is not a trace, since $$\phi(Y_{e,f^{o}}Y_{e,f^{o}}^*) = \phi(Y_{e,f^{o}}^*Y_{e,f^{o}}) \frac{\mu(s(e))}{\mu(s(f))}.$$ In particular, the modular group of $\phi$ is non-trivial.

\begin{lemma} Let $\gamma = (e_1, f_1^o, e_2, f_2^o,\cdots,e_m, f_m^o)$ be a path in $\Gamma$.  Let $$Y_\gamma = 
\left(\prod_{j=1}^m \mu_{s(e_j)}\mu_{s(f_j)} \right)^{1/4} Y_{e_1,f_1^o} \cdots Y_{e_m,f_m^o}.$$ Then:\\
(a) The map $\pi : \gamma \mapsto Y_\gamma$ is a $*$-homomorphism from $Gr_0^\Gamma$ to $\ms B_\Gamma$ whose image is contained in the centralizer of the free-quasi free state $\phi$;\\
(b) Let $N= \{(e,f^o):e,f\in E_+, t(e)=t(f)\}$, and identify ${\ms B_\Gamma}^N$ with the space of maps from $N$ to $\ms B_\Gamma$.  Then the map $$\pi_1: \gamma \mapsto  \frac{\mu_{s(f_m)}}{\left[ \prod_{j=1^m} \mu_{s(e_j)}\mu_{s(f_j)} \right]^{1/4}} \delta_{(e_m,f_m^o)}  Y_{(e_1,f_1^o,\cdots,e_{m-1},f_{m-1}^o )}$$ is a linear map from $Gr_1$ into ${\ms B_\Gamma}^N$ (here $\delta_{(e,f)}$ denotes the delta function at $(e,f)\in N$);\\
(c) Regard $Gr_1\boxtimes Gr_1$ as a certain set of pairs of paths $(\gamma_1,\gamma_2)$.  Then the map $\pi_\boxplus = \pi_1\times \pi_1$ gives rise to a $*$-homomorphism from $Gr_1\boxtimes Gr_1$ into ${\ms B_\Gamma}^N\times {\ms B_\Gamma}^N \cong M_{\#N\times \#N}({\ms B_\Gamma}\otimes {\ms B_\Gamma}^{op})$
\\
(d) If $\tau$ is the 2-cabled Voiculescu trace then $\phi\circ \pi = \tau$ and $\frac{1}{\#N}Tr \circ \phi\otimes\phi^{op} \circ \pi_\boxtimes = \tau_1\boxtimes \tau_1$. 
\end{lemma}

Once again, we can define various free analysis maps on $\mc B_\Gamma$; these coincide with the ones considered in \cite{nelson:qf}.

Difference quotient: 
$$
 \partial_{(e,f^o)}[X_{e_1,f_1^o}\dotsb X_{e_m,f_m^o}] = \sum_{(e_j,f_j^o) = (e,f^o)} X_{e_1,f_1^o}\dotsb X_{e_{j-1},f_{j-1}^o}\otimes X_{e_{j+1},f_{j+1}^o}\dotsb X_{e_m,f_m^o};$$
 
Cyclic partial derivative:
$\ms D_{(e,f^o)}: \ms B_\Gamma \to \ms B_\Gamma$ by 
\begin{equation*}
 \ms D_{(e,f^o)}[X_{e_1,f_1^o}\dotsb X_{e_m,f_m^o}] = \mu_{s(e_1)} \cdot \sum_{(e_j,f_j^o) = (f,e^o)} X_{e_{j+1},f_{j+1}^o}\dotsb X_{e_m,f_m^o} X_{e_1,f_1^o}\dotsb X_{e_{j-1},f_{j-1}^o}
\end{equation*}

Cyclic gradient:
$$\ms D(G) = (\ms D_{e,f^o}(G))_{(e,f^o)} \in (\ms A_\Gamma)^N.$$

Jacobian:
$$\ms J:(\ms A_\Gamma)^N \to M_N(\ms A_\Gamma \otimes \ms A_\Gamma^{op}),\qquad (\ms J Y)_{(e,f^o),(g,h^o)} = \partial_{(g,h^o)}(Y_{e,f^o}).$$

For $q = \sum a_i \otimes b_i \in \ms A_\Gamma \otimes \ms A_\Gamma^{op}$ and $g \in \ms A$ define 
\begin{equation*}
 q \# g = \sum_i a_i g b_i.
\end{equation*}
For $q = (q_{ij}) \in M_{N}(\ms A_\Gamma \otimes \ms A_\Gamma^{op})$, and $g = (g_j)$, $h = (h_j) \in \ms A_\Gamma^N$ define
\begin{align*}
 (q \# g)_j &= \sum_i q_{ji} \# g_i\\
 g \cdot h &= \sum_i g_i h_i
\end{align*}

\begin{proposition}\label{graph-derivatives-nt}
The maps $\pi$, $\pi_1$ and $\pi_\boxplus$ are equivariant with respect to the maps $\ms D: Gr_0\to Gr_1$ and $\ms D$, $\ms \partial: Gr_1\to Gr_1\boxplus Gr_1$ and $\ms J$, $\#: Gr_1 \boxtimes Gr_1 \times Gr_1 \to Gr_1$ and $\cdot$.
\end{proposition}


\subsection{Application: Factoriality of \texorpdfstring{$M_1=W^*(Gr_1)$}{M1=W*(Gr1)}}

\begin{lemma}\label{lemma:likeCoarse}
Assume that $\Phi^*(\tau)<\infty$.  Then: (a) the spectral measure of $\cup$ is non-atomic, and (b)  if  $Q\in L^2(Gr_k \boxtimes Gr_k^{op})$ satisfies
$$
\begin{array}{c}
\begin{tikzpicture}[scale=.5]
 \draw[Box] (0,0) rectangle (2,1); \node at (1,.5) {$Q$}; 
 \node[marked,scale=.8,above left] at (0,1) {};
 \draw[thickline] (1,1) --++(0,1);
 \draw[thickline] (1,0)--++(0,-1);
\draw[thickline](0,.5) --node[rcount,scale=.75] {$2k$} ++(-1.5,0);
   \draw[thickline] (2,.5)--node[rcount,scale=.75] {$2k$} ++(1.5,0);
 \node[above left] at (-0.2,0.9){$\bigcup$};

\end{tikzpicture}
\end{array}
=
\begin{array}{c}
\begin{tikzpicture}[scale=.5]
 \draw[Box] (0,0) rectangle (2,1); \node at (1,.5) {$Q$}; 
 \node[marked,scale=.8,above left] at (0,1) {};
 \draw[thickline] (1,1) --++(0,1);
 \draw[thickline] (1,0)--++(0,-1);
  \draw[thickline](0,.5) --node[rcount,scale=.75] {$2k$} ++(-1.5,0);
  \draw[thickline] (2,.5)--node[rcount,scale=.75] {$2k$} ++(1.5,0);
 \node[above left] at (-0.2,-1.2){$\bigcap$};

\end{tikzpicture}
\end{array}
$$ then $Q=0$.
\end{lemma}
\begin{proof}
If $\Phi^*(\xi)<\infty$, then the conjugate variable $\xi$ exists in $L^2(Gr_1)$.  Denote by $x$ the image of $\cup$ under the unital trace-preserving embedding of $Gr_0$ into $Gr_1$.  Then 
$$\tau_1(\delta \xi x^n ) = \delta \tau_1\boxtimes\tau_1 \left(\sum_{k=1}^n \supset 
{\genfrac{}{}{0pt}{}{\cup^{k-1}}{\cap^{n-k}}}\subset \right) =  \sum_{k=1}^n \tau_1\otimes\tau_1 (x^{k-1}\otimes x^{n-k}).$$
It follows that if we denote by $\xi'$ the orthogonal projection of $\delta \xi$ onto $L^2(W^*(x))$, then $\xi'$ is exactly the conjugate variable to $x$.  It follows that $\Phi^*(x)<\infty$, and so $x$ (and thus $\cup$) has diffuse spectral measure by \cite{freefisher}.  This proves part (a).

To prove part (b), let use use the notation $M_0 = W^*(Gr_0)$, $P = W^*(Gr_k \boxtimes Gr_k)$. Then $M_0\otimes M_0^{op} \subset P$ as a unital subalgebra, and in particular $L^2(P)$ as an $M_0\otimes M_0^{op}$-module is a multiple of $L^2(M_0\otimes M_0^{op})$.  

Now if $Q\in L^2(P)$ satisfies the hypothesis of part (b) of the Lemma, then $(\cup\otimes 1-1\otimes\cup)Q=0$.  But since $L^2(P)$ is a multiple of the coarse $M_0\otimes M_0^{op}$ bimodule, if $Q\neq 0$ this would imply the existence of a nonzero element of $L^2(M_0 \otimes M_0^{op})$ which is annihilated by $(\cup\otimes1-1\otimes \cup)$.  But this would mean that a non-zero Hilbert-Schmidt operator commutes with the operator $\cup$, which is impossible since by part (a) the operator $\cup$ has diffuse spectrum.
\end{proof}

\begin{lemma}\label{lemma:closable}
Assume that $\Phi^*(\tau)<\infty$. Then $\partial$ is a closable derivation densely defined on $L^2(Gr_1)$ with values in $L^2(Gr_1\boxtimes Gr_1)$.  The domain of the adjoint $\partial^*$ includes $Gr_1\boxtimes Gr_1$. 
\end{lemma}

\begin{proof}

Recall that $T_l$ represents the result of applying $\tau$ to a box with $2l$ strings.  It is a straightforward verification that if $Q\in Gr_1\boxtimes Gr_1^{op}$, then  $\partial^*(Q)$ is given by the following expression:
$$
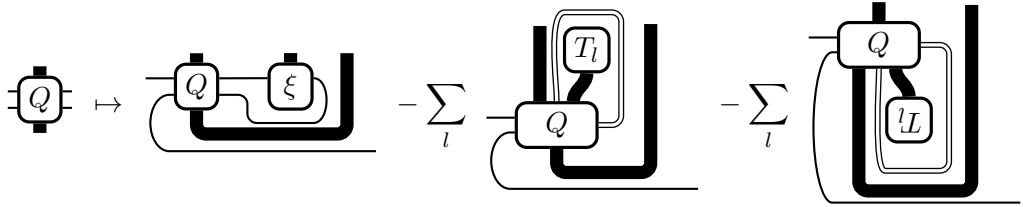
\begin{array}{c}\begin{tikzpicture}[scale=.5]
\node[draw,Box,rectangle] at (0.5,0.5)(Q){$Q$};
 \draw[thick](Q.160 )--++(-0.25,0);
  \draw[thick](Q.200)--++(-0.25,0);
 \draw[thick](Q.20)--++(0.25,0);
 \draw[thick](Q.340)--++(0.25,0);
 \draw[thickline](Q.north)--++(0,0.25);
 \draw[thickline](Q.south)--++(0,-0.25);
 \end{tikzpicture}\end{array}
 \mapsto 
 \begin{array}{c}\begin{tikzpicture}[scale=.5]
\node[draw,Box,rectangle,minimum width=1] at (0.5,0.5)(Q){$Q$};
\node[draw,Box,rectangle,minimum width=1] at (3,0.5)(xi){$\ \!\xi\ \! $};
\draw[thickline] (Q.north) --++(0,0.25);
\draw[thickline] (xi.north) --++(0,0.25);
\draw[thick](Q.160) --++(-0.75,0);
\draw[thick](Q.200) --++(-0.125,0)  arc (-270:-90:0.5 and 0.75) -- ++(5.5,0);
\draw[thick](Q.20) -- (xi.160);
\draw[thick](Q.340) --++(0.5,0) arc(90:0:.25) --++(0,-0.25) arc(0:90:-0.25) -- ++(1.5,0) to[in=0,out=0] (xi.20) ;
\draw[thickline](Q.south) --++(0,-0.375) arc(0:90:-0.25) -- ++(3.5,0) arc(90:180:-0.25) --++(0,1.9);
  \end{tikzpicture}\end{array}
 - \sum_{l} 
 \begin{array}{c}\begin{tikzpicture}[scale=.5]
\node[draw,Box,rectangle,minimum width=1] at (0.5,0.5)(Q){$\ \ Q\ \ $};
\draw[thick](Q.170) --++(-0.75,0);
\draw[thick](Q.190) --++(-0.125,0)  arc (-270:-90:0.5 and 0.75) -- ++(5,0);
\draw[line width=.5pt,double,double distance=1.25pt](Q.0) --++(0.35,0) arc(90:180:-0.25) --++(0,2.5) arc(-180:-90:-0.25) --++(-1.25,0) to [in=90,out=180,looseness=0.5] (Q.90) ;

\node[Box,rectangle,draw] at (1.3,2.5) (tau){$T_l$};
\draw[thickline](Q.55) to[in=270,out=90] (tau.south); 
\draw[thickline](Q.south) --++(0,-0.375) arc(0:90:-0.25) -- ++(2,0) arc(90:180:-0.25) --++(0,3.7);
\draw[thickline](Q.125) --++(0,2);
  \end{tikzpicture}\end{array}
-\sum_l  
\begin{array}{c}\begin{tikzpicture}[scale=.5]
\node[draw,Box,rectangle,minimum width=1] at (0.5,0.5)(Q){$\ \ Q\ \ $};
\draw[thick](Q.170) --++(-0.75,0);
\draw[thick](Q.190) --++(-0.125,0)  arc (-270:-90:0.5 and 2) -- ++(5,0);
\draw[line width=.5pt,double,double distance=1.25pt](Q.0) --++(0.5,0) arc(90:0:0.25) --++(0,-2.85) arc(0:-90:0.25) --++(-1.45,0) to [in=270,out=180,looseness=0.3] (Q.270) ;

\node[Box,rectangle,draw,rotate=180] at (1.3,-1.5) (tau){$T_l$};
\draw[thickline](Q.305) to[in=90,out=270] (tau.south); 
\draw[thickline](Q.230) --++(0,-3) arc(0:90:-0.25) -- ++(2.5,0) arc(90:180:-0.25) --++(0,4.7);
\draw[thickline](Q.north) --++(0,0.5);
  \end{tikzpicture}\end{array}
$$
Since the domain of $\partial^*$ is dense, it follows that $\partial$ is closable.
\end{proof}

Assume now that $\delta > 1$.  Then the elements $\jones$ and $\twostrings$ of $P_2$ (with all of the two possible shadings) are all different (since they are different in $TL_2(\delta)\subset P_2$).  It follows that we can find an element of $P_2$ in the linear span of these two elements, but which is perpendicular to $\jones$ (with all of its possible shadings).  In fact, since $TL^+_2(\delta)$ is two-dimensional and is closed under multiplication $\wedge_2$ of $Gr_2$, we may choose the element to be idempotent.  We will denote this element by $\jw$ (this is the so-called Jones-Wenzl idempotent).  

By our choices, we thus have identities  $\jw \hskip -0.5pt \jw = \jw$ and $\sidecup\hskip -0.5pt \jw=0$.

Let $\partial': Gr_1\to Gr_1\boxtimes Gr_1$ be given by the formula $$\partial'(x) = \partial(x) \cdot \jw.$$ 
In other words,
\begin{equation}\label{eq:partialPrime}
\begin{tikzpicture}[scale=.5]
 \draw[Box] (0,0) rectangle (2,1); \node at (1,.5) {$x$};\node[marked,scale=.8,above left] at (0,1) {};
 \draw[thick] (0,.5) --++(-.5,0);
 \draw[thick] (2,.5) arc(90:-90:.3 and .375) --++(-2.5,0);
 \draw[thickline] (1.5,1) arc(180:0:.7cm and .5cm) -- node[rcount,scale=.75] {$2k$} ++(0,-2);
 \draw[thickline] (.5,1) --++(0,1.5);
 \draw[thick] (1.15,1) arc(180:0:1.2cm and .85cm) -- ++(0,-.75) arc(180:270:.5cm) --++(.1,0) arc(-90:-45:.25) -- (4.75,.25);
 \draw[thick] (.85,1) arc(180:0:1.5cm and 1.1cm)  arc(180:270:.25cm) --++(.1,0) arc(90:45:.25) -- (4.75,.25);
 
 \draw[thick] (5.5,-.25) --++ (-.25,0) arc(270:225:.25) -- (4.75,.25);
 \draw[thick] (5.5,0.75) --++ (-.25,0) arc(90:135:.25) -- (4.75,.25);

 \node[left] at (-1,.5) {$\partial'(x)=\displaystyle\sum_{k} $};
\end{tikzpicture}
\end{equation}

Note that because $\sidecup\hskip -0.5pt \jw=0$, $\partial'(\cup)=0$.  Furthermore, under the hypothesis of Lemma~\ref{lemma:closable},  we see that the domain of the adjoint of $\partial'$ again includes $Gr_1\boxtimes Gr_1$, and so $\partial'$ is again closable.

\begin{lemma}\label{lemma:commutatorWithCup}
Assume that $\delta >1$ and $\Phi^*(\tau)<\infty$.  Denote by $Y$ the image of $\cup$ in $Gr_1$. If $Z\in W^*(Gr_1,\tau)$ satisfies $[Z,Y]=0$, then $Z$ is in the domain of the closure of $\partial'$ and $\overline{\partial'}(Z)=0$.
\end{lemma}
\begin{proof}
Since by Lemma~\ref{lemma:closable} $\partial$ (and thus $\partial'$) is a closable derivation, we may apply the theory of Dirichlet forms as in \cite{dabrowski:Fisher,peterson:ell2rigidity}.  We follow \cite{dabrowski:Fisher} and adopt notations close to the ones used in that paper.

We set $\Delta = (\partial')^*\overline{\partial'}$ and consider $\eta_\alpha = \alpha(\alpha + \Delta)^{-1}$, for $\alpha >0$.  We also set $$\zeta_\alpha = \eta_\alpha^{1/2} = \pi^{-1} \int_0^\infty \frac{t^{-1/2}}{1+t}\eta_{\alpha(1+t)/t} dt$$ and $\partial_\alpha = \partial'\circ \zeta_\alpha$.  Because $\partial'$ is a self-adjoint derivation, it follows that  $\phi_t = \exp(-\Delta t)$ is a semigroup of completely-positive maps. As a consequence, both  $$\eta_\alpha =  \alpha \int_0^\infty \exp(-\alpha_t) \phi_t dt$$ and $\zeta_\alpha$ are completely-positive.   They also have a regularizing effect on $\partial'$: $\partial_\alpha$ is a bounded map; moreover, $\Vert \zeta_\alpha(x) - x \Vert_2 \to 0$ for all $x$.

Since $Y\in \ker \partial' \subset \ker \Delta$, one can easily get that $0=\partial_\alpha ( [Z,Y] ) = [ \partial_\alpha(Z),Y ]$.  But by Lemma~\ref{lemma:likeCoarse}(b), we conclude that $\partial_\alpha(Z) = 0$ for all $\alpha>0$.  Taking the limit $\alpha\to 0$ and using that $\Vert Z-\zeta_\alpha(Z)\Vert_2 \to 0$ and closability of $\partial'$, we deduce that $Z$ is in the domain of the closure $\overline{\partial'}$ of $\partial'$ and that $\overline{\partial'}(Z)=0$.
\end{proof}

\begin{lemma}  \label{lemma:commutatorWithJW} Let $\delta>1$ and assume that $\Phi^*(\tau)<\infty$.  Assume that $Z\in W^*(Gr_1)$ is in the domain of $\overline{\partial'}$ and $\overline{\partial'}(Z)=0$.  Assume further that $[Z,
\jwcup
]=0$.  Then $Z$ is a multiple of $1$.  
\end{lemma}
\begin{proof}
 Using $\overline{\partial'}(Z)=0$, $\jw \hskip -0.5pt \jw = \jw$  and the Leibnitz rule, we compute: $$0=\overline{\partial'}([Z,
\jwcup
])= 
(Z\otimes 1 - 1\otimes Z^{op}) (\jw)
 $$
  
Note that because for $\delta > 1$,  $
\twostrings
 $ is not a multiple of $\jones$ but is in the linear span of $\jw$ and $\jones$, the orthogonal projection  $P$ of $\jw$ onto $L^2(Gr_1 \otimes Gr_1^{op})\subset L^2(Gr_1 \boxtimes Gr_1^{op})$ is nonzero multiple of $\twostrings=1\otimes 1$.  It follows that $0=Z\otimes 1-1\otimes Z^{op} \in L^2(Gr_1 \otimes Gr_1^{op})$.  But this implies that $Z$ is a scalar multiple of identity.
\end{proof}

Combining these lemmas we have:

\begin{theorem} \label{thm:factoriality}
Assume that $\delta > 1$ and $\Phi^*(\tau)<\infty$.  Then $M_1=W^*(Gr_1,\tau)$ is a factor.  
\end{theorem}

\subsection{Application: Higher relative commutants}

It will be convenient to consider the spaces $R_k = \bigoplus_{m_1,m_2 \textrm{ even}} P_{k+1+\frac{1}{2}(m_1+m_2)}$ regarded as modules over $Gr_k(P)\otimes Gr_k(P)^{op}$ as follows.  For $Q\in R_k$, regard $Q$ as a diagram drawn as follows:
$$
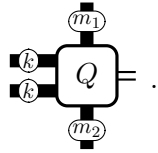
\begin{array}{c}
\begin{tikzpicture}[scale=.5]
\draw[Box] (0,0) rectangle (1.6,1.6); \node at (.8,.8){$Q$};
\draw[thickline](0.8,1.6) --++(0,.15) -- node[rcount,scale=.75] {$m_1$} ++(0,1);
\draw[thickline](0.8,0) --++(0,-0.15) -- node[rcount,scale=.75] {$m_2$} ++(0,-1);
\draw[thickline](0,1.2) --++(-0.25,0) -- node[rcount,scale=.75] {$k$} ++(-1,0);
\draw[thickline](0,0.4) --++(-0.25,0) -- node[rcount,scale=.75] {$k$} ++(-1,0);
\draw[thick](1.6,0.9) --++(0.5,0);
\draw[thick](1.6,0.7) --++(0.5,0);
 \end{tikzpicture}
 \end{array}.
$$

Note that $$
\begin{array}{c}
\begin{tikzpicture}[scale=.5]
\draw[Box] (0,0) rectangle (1.6,1.6); \node at (.8,.8){$Q$};
\draw[thickline](0.8,1.6) --++(0,.25);
\draw[thickline](0.8,0) --++(0,-0.25);
\draw[thickline](0,1.2) --++(-0.25,0) -- node[rcount,scale=.75] {$k$} ++(-1,0);
\draw[thickline](0,0.4) --++(-0.25,0) -- node[rcount,scale=.75] {$k$} ++(-1,0);
\draw[thick](1.6,0.9) --++(0.5,0);
\draw[thick](1.6,0.7) --++(0.5,0);
 \end{tikzpicture}
 \end{array}
 \mapsto \delta^{-(k-1)}
\begin{array}{c}
\begin{tikzpicture}[scale=.5]
\draw[Box] (0,0) rectangle (1.6,1.6); \node at (.8,.8){$Q$};
\draw[thickline](0.8,1.6) --++(0,.25);
\draw[thickline](0.8,0) --++(0,-0.25);
\draw[thickline](0,1.2) --++(-0.25,0) -- node[rcount,scale=.75] {$k$} ++(-1,0);
\draw[thickline](0,0.4) --++(-0.25,0) -- node[rcount,scale=.75] {$k$} ++(-1,0);
\draw[thick](1.6,1.5) --++(2.4,0);
\draw[thick](1.6,0.1) --++(2.4,0);
\draw[thickline](4,.375) --++(-1,0) arc(270:180:0.2) -- node[rcount,scale=.75]{$k-1$} ++(0,0.45) arc(-180:-270:.2) --++(1,0);
 \end{tikzpicture}
 \end{array}
 $$
 embeds $R_k$ into $Gr_k\boxtimes Gr_k$.  We endow $R_k$ with inner product structure inherited from the inner product on $Gr_k\boxtimes Gr_k^{op}$.  We also endow it with the left action of $Gr_k\otimes Gr_k^{op}$ coming from the embedding of $Gr_k\otimes Gr_k^{op}$ into $Gr_k\boxtimes Gr_k^{op}$.  Explicitly, for $a\in Gr_k$, $b\in Gr_k^{op}$, we have
 $$
 (a\otimes b) \cdot \!\!\!
\begin{array}{c}
\begin{tikzpicture}[scale=.5]
\draw[Box] (0,0) rectangle (1.6,1.6); \node at (.8,.8){$Q$};
\draw[thickline](0.8,1.6) --++(0,.25);
\draw[thickline](0.8,0) --++(0,-0.25);
\draw[thickline](0,1.2) --++(-0.25,0) -- node[rcount,scale=.75] {$k$} ++(-1,0);
\draw[thickline](0,0.4) --++(-0.25,0) -- node[rcount,scale=.75] {$k$} ++(-1,0);
\draw[thick](1.6,0.9) --++(0.5,0);
\draw[thick](1.6,0.7) --++(0.5,0);
 \end{tikzpicture}
 \end{array}
=
\begin{array}{c}
\begin{tikzpicture}[scale=.5]
\draw[Box] (0,0) rectangle (1.6,1.6); \node at (.8,.8){$Q$};
\draw[thickline](0.8,1.6) --++(0,.25);
\draw[thickline](0.8,0) --++(0,-0.25);
\draw[thickline](0,1.2)  --++(-1,0);
\draw[thickline](0,0.4) --++(-1,0);
\draw[Box](-2,1.5) rectangle (-1,0.9); \node at (-1.5,1.25){$a$};
\draw[thickline](-1.5,1.5) --++(0,0.25);
\draw[thickline](-2,1.25) --++(-0.5,0);
\draw[Box](-2,0.7) rectangle (-1,0.1); \node[rotate=180] at (-1.5,.4){$b$};
\draw[thickline](-1.5,0.1) --++(0,-0.25);
\draw[thickline](-2,.4) --++(-0.5,0);
\draw[thick](1.6,0.9) --++(0.5,0);
\draw[thick](1.6,0.7) --++(0.5,0);
 \end{tikzpicture}
 \end{array}.
$$
We will denote by $R_k$ also the completion of $R_k$ with respect to its inner product.

Set  
$$\Xi_k = \begin{array}{c}
\begin{tikzpicture}[scale=.5]
\draw[thickline] (0,0) arc (270:360:0.75 and 0.4) -- node[rcount,scale=.75] {$k-1$} ++(0,1)
arc(0:90:0.74 and 0.4);
\draw[thick] (0,2.1) --++ (1.5,0);
\draw[thick] (0,-.3) --++ (1.5,0);
\end{tikzpicture}
\end{array}\in R_k.
$$

Let $A_{k-1}$ be the subalgebra of $Gr_k$ generated by all diagrams of the form 
$$
\left\{
\begin{array}{c}
\begin{tikzpicture}[scale=.5]
\draw[Box](1,1) rectangle (0,0); \node at (0.5, 0.5){$Q$};
\draw[thickline](0,0.5) --++(-0.5,0);
\draw[thickline](1,0.5) --++(0.5,0);
\draw[thick] (-0.5,1.3) --++ (2,0);
\end{tikzpicture}
\end{array}
:Q\in P_{k-1} \right\}.
$$

Consider the map $Gr_k \otimes Gr_k^{op} \to R_k$ given by $$(X\otimes Y)\mapsto (X\otimes Y)\cdot \Xi_k.$$  It is not hard to see that this map descends to an isometry from $L^2(Gr_k)\otimes_{A_{k-1}} L^2(Gr_k)$ into $R_k$. We denote the image of this map by $V$.

\begin{lemma} \label{lemma:coarseOverA} Let  $A_{k-1}, \Xi_k$ and $V$ be as above.  Let $$\Xi'_k = \begin{array}{c}
 \begin{tikzpicture}[scale=.5]
 \draw[thickline] (0,0) arc (270:360:0.75 and 0.4) -- node[rcount,scale=.75] {$k-1$} ++(0,1)
arc(0:90:0.74 and 0.4);
\draw[thick] (0,2.1) --++ (1,0) arc(90:45:.25) --++(2.25,-2.25) arc(45:90:-.25) --++(0.25,0) ;
\draw[thick] (0,-.3) --++ (1,0) arc(-90:-45:.25) --++(2.25,2.25) arc(135:90:.25) --++(0.25,0);
\end{tikzpicture}
%
 \end{array} \in R_k.
 $$
Assume that $(Z\otimes 1-1\otimes Z^{op})\cdot \Xi_k'=0$ for some $Z\in W^*(Gr_k)$, $\delta>1$ and $\Phi^*(\tau)<\infty$.  Then $Z\in A_{k-1}$.
\end{lemma}
\begin{proof}
Consider the $Gr_k\otimes Gr_k^{op}$ linear projection of $V$ onto the closure of $(Gr_k\otimes Gr_k^{op})\cdot \Xi'_k$ .  Then the projection of $(Z\otimes 1-1\otimes Z^{op})\cdot \Xi_k$ is a nonzero multiple of $(Z\otimes 1-1\otimes Z^{op})\Xi'_k$, since the projection $\Xi_k$ onto $\Xi'_k$ is nonzero.  It follows that also $(Z\otimes 1-1\otimes Z^{op})\Xi_k=0$, which means that $(Z\otimes 1-1\otimes Z)$ is zero in $L^2(Gr_k \otimes_{A_{k-1}} Gr_k)$.  But since $A_{k-1}$ is finite-dimensional and $Gr_k$ is diffuse (containing the diffuse element $\cup$), this implies that $Z\in A_{k-1}$.
\end{proof}

\begin{theorem} \label{thrm:highercomm} Let $A_{k-1}$ be as above.  Assume that $\delta>1$, $\Phi^*(\tau)<\infty$ and  $Z\in M_1'\cap M_k$.  Then $Z\in A_{k-1}$.
\end{theorem}
\begin{proof}
Let us consider the derivation $\partial_k: Gr_k \to R_k$ given by 
$$
\begin{tikzpicture}[scale=.5]
 \draw[Box] (0,0) rectangle (2,1); \node at (1,.5) {$x$};\node[marked,scale=.8,above left] at (0,1) {};
 \draw[thickline] (0,.5) --node[rcount,scale=.75] {$k$}++(-1,0);
 \draw[thickline] (2,.5) arc(90:-90:.3 and .575) -- node[rcount,scale=.75] {$k$} ++(-3,0);
 \draw[thickline] (1.5,1) arc(180:0:.7cm and .5cm) -- node[rcount,scale=.75] {$2l$} ++(0,-2);
 \draw[thickline] (.5,1) --++(0,1.5);
 \draw[thick] (1.15,1) arc(180:0:1.2cm and .85cm) -- ++(0,-.75) arc(180:270:.5cm) --++(.3,0);
 \draw[thick] (.85,1) arc(180:0:1.5cm and 1.05cm) -- ++(0,-.85) arc(180:270:.5cm and .15cm);
 \node[left] at (-1,.5) {$x \mapsto \displaystyle\sum_{l \geq 0} $};
\end{tikzpicture}
$$

If $\xi\in L^2(Gr_1)$ is the conjugate variable, consider its image in $L^2(Gr_k)$ under the trace-preserving embedding of $L^2(Gr_1)\subset L^2(Gr_k)$.  We then see that $$
\tau_k(\xi_k x) = \tau_k\boxtimes \tau_k (\partial_k(x) \Xi_k^*).
$$

One can once again argue that $\partial_k$ is closable; its adjoint is densely defined by the formula
$$
\begin{array}{c}\begin{tikzpicture}[scale=.5]
\node[draw,Box,rectangle] at (0.5,0.5)(Q){$Q$};
 \draw[thickline](Q.160 )--++(-0.25,0);
  \draw[thickline](Q.200)--++(-0.25,0);
 \draw[thick](Q.20)--++(0.25,0);
 \draw[thick](Q.340)--++(0.25,0);
 \draw[thickline](Q.north)--++(0,0.25);
 \draw[thickline](Q.south)--++(0,-0.25);
 \end{tikzpicture}\end{array}
 \mapsto 
 \begin{array}{c}\begin{tikzpicture}[scale=.5]
\node[draw,Box,rectangle,minimum width=1] at (0.5,0.5)(Q){$Q$};
\node[draw,Box,rectangle,minimum width=1] at (3,0.5)(xi){$\ \!\xi\ \! $};
\draw[thickline] (Q.north) --++(0,0.25);
\draw[thickline] (xi.north) --++(0,0.25);
\draw[thickline](Q.160) --++(-0.75,0);
\draw[thickline](Q.200) --++(-0.125,0)  arc (-270:-90:0.5 and 0.75) -- ++(5.5,0);
\draw[thick](Q.20) -- (xi.160);
\draw[thick](Q.340) --++(0.5,0) arc(90:0:.25) --++(0,-0.25) arc(0:90:-0.25) -- ++(1.5,0) to[in=0,out=0] (xi.20) ;
\draw[thickline](Q.south) --++(0,-0.375) arc(0:90:-0.25) -- ++(3.5,0) arc(90:180:-0.25) --++(0,1.9);
  \end{tikzpicture}\end{array}
 - \sum_{l} 
 \begin{array}{c}\begin{tikzpicture}[scale=.5]
\node[draw,Box,rectangle,minimum width=1] at (0.5,0.5)(Q){$\ \ Q\ \ $};
\draw[thickline](Q.170) --++(-0.75,0);
\draw[thickline](Q.190) --++(-0.125,0)  arc (-270:-90:0.5 and 0.75) -- ++(5,0);
\draw[thick](Q.355) --++(0.5,0) arc(90:180:-0.25) --++(0,2.85) arc(-180:-90:-0.25) --++(-1,0) to [in=90,out=180] (Q.90) ;
\draw[thick](Q.5) --++(0.35,0) arc(90:180:-0.25) --++(0,2.5) arc(-180:-90:-0.25) --++(-.75,0) to [in=90,out=180] (Q.80) ;

\node[Box,rectangle,draw] at (1.3,2.5) (tau){$T_l$};
\draw[thickline](Q.55) to[in=270,out=90] (tau.south); 
\draw[thickline](Q.south) --++(0,-0.375) arc(0:90:-0.25) -- ++(2,0) arc(90:180:-0.25) --++(0,3.7);
\draw[thickline](Q.125) --++(0,2);
  \end{tikzpicture}\end{array}
-\sum_l  
\begin{array}{c}\begin{tikzpicture}[scale=.5]
\node[draw,Box,rectangle,minimum width=1] at (0.5,0.5)(Q){$\ \ Q\ \ $};
\draw[thickline](Q.170) --++(-0.75,0);
\draw[thickline](Q.190) --++(-0.125,0)  arc (-270:-90:0.5 and 2.1) -- ++(5,0);
\draw[thick](Q.355) --++(0.35,0) arc(90:0:0.25) --++(0,-2.5) arc(0:-90:0.25) --++(-.75,0) to [in=270,out=180] (Q.280) ;
\draw[thick](Q.5) --++(0.5,0) arc(90:0:0.25) --++(0,-2.85) arc(0:-90:0.25) --++(-.95,0) to [in=270,out=180] (Q.270) ;

\node[Box,rectangle,draw,rotate=180] at (1.3,-1.5) (tau){$T_l$};
\draw[thickline](Q.305) to[in=90,out=270] (tau.south); 
\draw[thickline](Q.230) --++(0,-3) arc(0:90:-0.25) -- ++(2.5,0) arc(90:180:-0.25) --++(0,4.7);
\draw[thickline](Q.north) --++(0,0.5);
  \end{tikzpicture}\end{array}.
$$

We also introduce $\partial'_k:Gr_k\to R_k$ given by the formula
\begin{equation*}
\begin{tikzpicture}[scale=.5]
 \draw[Box] (0,0) rectangle (2,1); \node at (1,.5) {$x$};\node[marked,scale=.8,above left] at (0,1) {};
 \draw[thickline] (0,.5) --++(-.5,0);
 \draw[thickline] (2,.5) arc(90:-90:.3 and .575) --++(-2.5,0);
 \draw[thickline] (1.5,1) arc(180:0:.7cm and .5cm) -- node[rcount,scale=.75] {$2l$} ++(0,-2);
 \draw[thickline] (.5,1) --++(0,1.5);
 \draw[thick] (1.15,1) arc(180:0:1.2cm and .85cm) -- ++(0,-.75) arc(180:270:.5cm) --++(.1,0) arc(-90:-45:.25) -- (4.75,.25);
 \draw[thick] (.85,1) arc(180:0:1.5cm and 1.1cm)  arc(180:270:.25cm) --++(.1,0) arc(90:45:.25) -- (4.75,.25);
 
 \draw[thick] (5.5,-.25) --++ (-.25,0) arc(270:225:.25) -- (4.75,.25);
 \draw[thick] (5.5,0.75) --++ (-.25,0) arc(90:135:.25) -- (4.75,.25);

%
 
 \node[left] at (-1,.5) {$\partial_k'(x)=\displaystyle\sum_{l \geq 0} $};
\end{tikzpicture}
\end{equation*}

Let us now write $M_k = W^*(Gr_k,\tau)$ and assume that $Z\in M_1'\cap M_k$.  Then since $Z$ commutes with the image of $\cup$ in $Gr_k$, we conclude by a similar approximation argument that $\overline{\partial'_k}(Z)=0$.  Next, we compute $\overline{\partial'_k}([Y,Z])$ where $Y$ is the image of the element $\jwcup$ in $Gr_k$.  We thus conclude that $$
 0 = \overline{\partial'_k}([Y,Z]) = (Z\otimes 1-1\otimes Z^{op})\cdot \Xi'_k.  
 $$

But then $Z\in A_{k-1}$ by Lemma~\ref{lemma:coarseOverA}.
\end{proof}

As a corollary, we conclude that $M_k$ is a factor for $k\geq 1$.  Indeed, $M_k'\cap M_k \subset M_1'\cap M_k =A_{k-1}$, and it is easily verified that only multiples of the identity element of $A_{k-1}$ commute with $M_k$. We also get that $M_0 = e_0 M_2 e_0$ (where $e_0 = \supset \subset$) is a factor.  Arguing exactly as in \cite{gjs1}, we obtain:

\begin{corollary}  \label{cor:standard} Assume that $\delta>1$ and $\Phi^*(\tau)<\infty$.  Then the standard invariant of the subfactor inclusion $M_0\subset M_1$ is the planar algebra $\mc P$.
\end{corollary}

It would be interesting to see if there is a similar statement in the setting of \cite{gjsz,nelson} (i.e., in the ``singly-cabled'' case).  In particular, this would give a new proof of factoriality, avoiding the rather ad-hoc techniques used in \cite{gjs1,jsw,bro1}.  Unfortunately, the ``singly-cabled'' differential calculus is harder to deal with than our setting.  

\subsection{Further applications: \texorpdfstring{$L^2$-rigidity and lack of Property $\Gamma$}{L2-rigidity and lack of Property Gamma}.}

\begin{lemma} \label{lem:partialDerivatives} Assume that $P_1^\pm = \mathbb C$. 
For $x,y\in Gr_1(P)$, view $x\otimes y^{op}\in Gr_1(P) \otimes Gr_1(P)^{op} \subset Gr_1(P)\boxtimes Gr_1(P)^{op}$, and let $\xi = \jones$ and $\zeta = \jw$, both viewed as elements in  $  L^2 (Gr_1(P) \boxtimes Gr_1(P))$.
Then
\begin{equation}\label{eq:cupsLike1}
\langle x\otimes y^{op} \xi ,\xi \rangle_{L^2(Gr_1(P)\boxtimes Gr_1(P)^{op})} =  \tau(x)\tau(y).
\end{equation}
If moreover $\delta > 1$, then for some nonzero $\omega$,
\begin{equation}\label{eq:jwlike1}
\langle x \otimes y^{op} \zeta, \zeta\rangle _{L^2(Gr_1(P)\boxtimes Gr_1(P)^{op})} = \omega \tau(x)\tau(y).
\end{equation}
In particular,  the  maps $i_1 : (x\otimes y)\mapsto (x\otimes y^{op})\jones$, $i_2 = (x\otimes y)\mapsto \omega^{-1/2} (x\otimes y^{op}) \jw$ extend to isometries from $L^2(Gr_1(P))\otimes L^2(Gr_1(P))$   into $L^2(Gr_1(P)\boxtimes Gr_1(P)^{op})$.
\end{lemma}
\begin{proof} 
Let $E_1$ denote the trace-preserving conditional expectation onto $P_1 \subset Gr_1(P)$ (viewed as diagrams with no strings on top).  Then
$$
\langle x\otimes y^{op} \xi ,\xi \rangle =  \tau_1( E_1(x) E_1(y) ).$$
But since $P_1=\mathbb{C}$, it follows that $E_1(x)= \tau_1(x)$, $E_1(y)=\tau_1(y)$, which shows \eqref{eq:cupsLike1}.

Next, note that \eqref{eq:jwlike1} holds with $\omega = 1$ if we were to replace $\zeta$ by $\twostrings$.  However,  $\twostrings$ is a linear combination of $\jw$ and $\jones$, which are orthogonal as vectors in $L^2(Gr_1(P)\boxtimes Gr_1(P)^{op})$.  It follows \eqref{eq:jwlike1} holds for some $\omega\geq 0$.  But setting $x=y=1$ gives us that $\omega$ = $\langle \jw,\jw\rangle > 0$ when $\delta > 1$.
%
\end{proof}

Let now $\partial'$ be as in \eqref{eq:partialPrime}, i.e., $\partial' (X) = \partial(X) \cdot \jw$, and let $\partial'' (X) = \partial(X) \cdot \jones$, for $X\in Gr_1$.  Let moreover $x_1$ be the image of $\cup$ in $Gr_1(P)$, and let $x_2 = \jwcup \in Gr_1(P)$.  

\begin{corollary}\label{cor:FDQ}  Assume that $P_1^\pm = \mathbb{C}$ and $\delta>1$.
With the above notation, the formulas $\partial_1=(i_1)^* \circ \partial''$ and $\partial_2  = (i_2)^*\circ \partial'$ define derivations on $Gr_1(P)$ with values in $L^2(Gr_1(P))\otimes L^2(Gr_1(P))$, satisfying 
\begin{equation}\label{eq:DiffQuotients}
\partial_i x_j = \delta_{i=j} 1\otimes 1.
\end{equation} 
Moreover, if $\Phi^*(\tau)<\infty$, then $\partial_i^*(1\otimes 1) \in L^2(Gr_1(P))$ and also  $\Phi^*(x_1,x_2)$ (in the sense of Voiculescu) is finite.
\end{corollary}
\begin{proof}
Equation \eqref{eq:DiffQuotients} is immediate from the definition of $\partial''$, $\partial'$ and the maps $i_1$, $i_2$, $e_1$, $e_2$. Since $\partial$ satisfies the Leibnitz rule and we used left multiplication on $Gr_1\boxtimes Gr_1^{op}$ to define $\partial'$ and $\partial''$, these maps are derivations; since $i_1$ and $i_2$ are bimodule maps, it follows that $\partial_i$ are derivations.   

If $\Phi^*(\tau)<\infty$, then the conjugate variable $\xi\in L^2(Gr_1(P))$ exists.  By Lemma~\ref{lemma:closable} and the fact that $p_i$ are given by diagrams, we see that $Gr_1\boxtimes Gr_1^{op}$ is in the domain of $\partial_i^*$.  It follows that $\partial^*_{i}$, $i=1,2$ are closable.  Moreover,  $\partial_1^*(1\otimes 1) = \omega^{-1/2}\partial^*(\jw)$ and $\partial_2^*(1\otimes 1)  = \partial'(\jones)$ exist in $L^2(Gr_1(P))$.  But then their projections onto  $L^2(W^*(x_1,x_2))\subset L^2(Gr_1(P))$ are exactly the conjugate variables $J(x_1:x_2)$ and $J(x_2:x_1)$ in the sense of Voiculescu \cite{freefisher}.  Thus $\Phi^*(x_1,x_2)<\infty$. 
\end{proof}

Recall the following definitions (see e.g. \cite{peterson:propertyT}): 

\begin{definition} Let $M$ be a II$_1$ factor, and let $F=\{u_1,\dots,u_k\}$ be a finite subset of $M$ consisting of unitaries.  \\
(a) We say that $F$ is a non-$\Gamma$ set for $M$ if there is a constant $K>0$ so that $$
\Vert \zeta - \langle \zeta, 1\rangle 1\Vert_{L^2(M)}^2 \leq K \sum_{j=1}^k \Vert u_j\zeta - \zeta u_j\Vert_{L^2(M)}^2,\qquad \forall \zeta \in L^2(M)$$
(b) We say that $F$ is a non-amenability set for $M$ if there is a constant $K>0$ so that 
$$\Vert \zeta \Vert_{L^2(M\otimes M^{op})}^2 \leq K \sum_{j=1}^k \Vert (u_j \otimes 1-1\otimes u_j^{op}) \zeta \Vert_{L^2(M\otimes M^{op})}^2,\qquad \forall \zeta \in L^2(M\otimes M^{op}).$$
\end{definition}

The following lemma is due to Connes \cite{cones:amenable} (see \cite{dabrowski-ioana} for a detailed proof):

\begin{lemma} \cite[Lemma 2.10]{dabrowski-ioana} Let $N$ be a II$_1$ factor and let $N_0\subset N$ be a weakly dense $C^*$-subalgebra.  If $N_0$ contains a non-$\Gamma$ set for N, it also contains a non-amenability set for $N$.
\end{lemma}

The following is one of the main results of \cite{dabrowski-ioana}:

\begin{theorem} \cite[Theorem 1.1]{dabrowski-ioana} \label{thm:dab-Io}
Let $(M,\tau)$ be a II$_1$ factor, and assume that $d:L^2(M,\tau)\to L^2(M,\tau)\bar\otimes L^2(M,\tau)^{\oplus \infty}$ is a densely defined real closed derivation with domain $D(d)$.  Assume that $d$ is not bounded, and $C^*(D(d)\cap M)$ contains a non-amenability set for $M$. Then $M$ is not $L^2$-rigid.  In particular, $M$ is prime, and does not have property $\Gamma$.
\end{theorem}

We record the following corollary:

\begin{corollary} \label{cor:L2rigid}
Assume that $(M,\tau)$ is a II$_1$ factor, and $X_1,X_2\in M$ are self-adjoint elements.  Assume that $d_1, d_2$ are densely defined closed derivations from $M$ to $Q=[L^2(M)\bar\otimes L^2(M)]^{\oplus \infty}$, so that: \\
(a) $X_k \in D(d_j)$ and $d_j X_k = \delta_{j=k} \omega_k\in Q$;  \\
(b) $ \langle a \omega_k b,\omega\rangle = \tau (a)\tau(b)$ for all $a,b\in W^*(X_1,X_2)$ and $k=1,2$.
(c) $ d_k^* (\omega_k) \in M$ and $d_k^* [1\otimes (d_k^* (\omega_k))] \in M$.  
\\
Then $M$ is not $L^2$-rigid; in particular, it is non-$\Gamma$ and does not have property $T$.
\end{corollary}
\begin{proof}
We let $d = d_1 \oplus d_2$. Let $N_0$ be the algebra generated by $X_1,X_2$, and let $N=W^*(N_0)$, $R_k = \overline{N \omega N}\subset Q$, $k=1,2$.  Then condition (b) implies that $R_k \cong L^2(N)\bar\otimes L^2 (N)$, and condition (a) implies that the restriction of $d_k$ to $N_0$ is the free difference quotient.  Moreover, $E_{N} (d_k^*(\omega _k))$ is then equal to the conjugate variable for $d_k$ (and thus $\Phi^*(X_1,X_2)<\infty$).  

Part (c) together with \cite{dabrowski:Fisher} implies that  $\{X_1,X_2\} \subset N_0$ is a non-$\Gamma$ set for $N$ (and thus a non-amenablity set for $N$, and thus also for $M$).  Thus $C^*(D(d))$ contains a non-amenability set for $M$.  Finally, because $\Phi^*(X_1)<+\infty$, $d_1$ is not bounded, so $d$ is not bounded either.  We may thus apply Theorem~\ref{thm:dab-Io} to conclude the proof.
\end{proof}

\begin{corollary}
Assume that $\tau_V$ is a $P$-trace, $\delta>1$ and assume that $\partial^*(\twostrings) \in Gr_1$ (this is the case, for example, for free Gibbs states).  Then $W^*(Gr_1)$ is not $L^2$-rigid, is non-$\Gamma$ and does not have property $(T)$.
\end{corollary}
\begin{proof}
We apply  Corollary~\ref{cor:FDQ} with $P=TL$ to deduce that the restrictions of the derivations $\partial'$ and $\partial''$ to the algebra generated by $x_1,x_2$ defined in that Corollary satisfies the conditions (a) and (b) of Corollary~\ref{cor:L2rigid}.  We see from Lemma~\ref{lemma:closable} that condition (c) is also satisfied, since $d_k^* \omega \in Gr_1$ and so also $1\otimes d_k^* \omega_k$ remains in the domain of $d_k^*$.
\end{proof}

\begin{proposition} \label{prop:criteriaNonGamma} Let $(M,\tau)$
  be a tracial probability space, $C\subset M$
  a dense subalgebra, $Q\subset C$
  a finite-dimensional subalgebra of $C$, and let $X_{1},X_{2}\in Q'\cap C$
  be self-adjoint variables. Assume that $\delta_{i}:C\to L^{2}(M)\otimes_{Q}L^{2}(M)$, $i=1,2$
  are derivations, so that $\delta_{i}(X_{j})=\delta_{i=j}1\otimes1$. Suppose that $\delta_{i}^{*}(1\otimes1)\in L^{2}(M)$,
   that $\delta_{i}$
  are closable, and that $W^{*}(X_{1},X_{2})$
  is a factor. 
  Assume that $(1\otimes E_{Q})\circ\delta_{i}:C\to L^{2}(M)$
  extends to a bounded map from $M$
  to $L^{2}(M)$. 
  
  Then for any free ultra-filter $\omega$, 
  $W^{*}(X_{1},X_{2})'\cap M^{\omega}\subset Q$.
  
  \end{proposition}

\begin{proof}
Let $B=W^*(X_1,X_2)$. Set $\hat \delta_k = \delta_k \big|_{\mathbb{C}[X_1,X_2]}$.
Then  
$\hat\delta_k^* (1\otimes 1) \in L^2(B)$.  

Since $B\subset Q'\cap M$, if $u,v\in B$ then $E_Q(u)\in Q'\cap Q$ and $E_Q(uv)=E_Q(vu)$. Thus for any extremal trace $\phi$ on $Q$, $\phi\circ E_Q$ is a trace on $B$, which is unique by factoriality of $B$ and thus does not depend on $\phi$.  Thus $E_Q(u)=\tau(u) 1$ for all $u\in B$.

The assumptions on $\delta_i$ thus guarantee that the restrictions of $\delta_i$ to the algebra generated by $X_1,X_2$ are exactly Voiculescu's free difference quotients.  It follows that $\Phi^*(X_1,X_2)$ is finite.  Thus by \cite[Theorem 13]{dabrowski:Fisher}, we see that $B$ is a factor that does not have property $\Gamma$.  In particular, $B$ is non-amenable.

  For $u\in C$, let $$\delta_P(u) = \sum_{i=1}^2 \delta_i (u) \# (X_i\otimes 1) - \delta_i(u)\#(1\otimes X_i) - [u,1\otimes 1],$$ where $ (a\otimes b)\#(c\otimes d) = ac\otimes db$ (which is well-defined 
  because $B$ commutes with $Q$).  
From the proof of \cite[Lemma 9]{dabrowski:Fisher}, one deduces that if $Z\in D(\overline{\delta_i})\cap M$, $i=1,2$, then 
\begin{eqnarray} \label{eq:distToScalars}
2\Vert Z- E_Q(Z)\Vert_2^2 &=&  \langle \delta_P(Z),[Z,1\otimes 1]\rangle_{L^2(M\otimes M)} 
\\ \nonumber && + \sum_{i=1}^2 \langle [Z,X_i],[Z,\delta_i^*(1\otimes 1)]\rangle_{L^2(M\otimes M)} 
\\ \nonumber && +\sum_{i=1}^2\operatorname{Re} \langle (E_Q\otimes 1-1\otimes E_Q)(\delta_i(Z),[Z,X_i]\rangle_{L^2(M)}.
\end{eqnarray}
By equation (1) in the the proof of the Free Poincare inequality (see \cite{dabrowski:Fisher}), we obtain that $\delta_P$ is a derivation vanishing on the algebra generated by $X_1$ and $X_2$.  Moreover, $$\delta_P^*(1\otimes 1) = \sum_{i=1}^2 [X_i, \delta_i^*(1\otimes 1)],$$ since for any $u\in C$ we have 
\begin{eqnarray*}
\langle \sum_{i=1}^2 [X_i, \delta_i^*(1\otimes 1)],u\rangle &=& \sum_{i=1}^2 \langle 1\otimes 1, \delta_i([X_i,u])\rangle 
\\ &=& \sum_{i=1}^2 \langle 1\otimes 1, [1\otimes 1,u]\rangle + \langle X_i\otimes 1-1\otimes X_i,\delta_i (u)\rangle 
\\ &=& \langle 1\otimes 1, \delta_P(u)\rangle.
\end{eqnarray*}
It follows that $\delta_P$ is closable and $W^*(X_1,X_2)$ lies in the domain of the closure; moreover, $\overline{\delta_P}(u)=0$ for any $u\in W^*(X_1,X_2)$.  

Let now $\Delta_P = \delta_P^* \overline{\delta_P}$ and let $\eta_\alpha = \alpha (\alpha + \Delta_P)^{-1}$ be completely positive contractions.

Since $Q$ is finite-dimensional, $L^2(Q)\subset (L^2(Q)\otimes L^2(Q))^p$ so that 
\begin{multline*}
L^2(M)\otimes _Q L^2(M) = L^2(M)\otimes_Q L^2(Q)\otimes_Q L^2(M) \\
\subset ((L^2(M)\otimes_Q L^2(Q) \otimes L^2(Q))\otimes _Q L^2(M))^p  = (L^2(M)\otimes L^2(M))^p.
\end{multline*}

Since $W^*(X_1,X_2)$ is non-amenable, it contains a non-amenablity set, we may thus find unitaries $u_1,\dots,u_k \in W^*(X_1,X_2)$ so that 
$$
\Vert \delta_P \eta_\alpha (Z_m)\Vert_2 \leq K \sum_{j=1}^k \Vert [\delta_P \eta_\alpha (Z_m),u_j]\Vert_2.
$$

Applying \eqref{eq:distToScalars} to $\eta_\alpha(Z_m)$ instead of $Z$ gives us the estimate
\begin{multline*}
2\Vert \eta_\alpha(Z_m) - E_Q(\eta_\alpha(Z_m))\Vert_2 \leq 2K 
\sum_{j=1}^k \Vert [\delta_P(\eta_\alpha(Z_m)),u_j]\Vert_2 \Vert Z_m \Vert_\infty 
\\ +\sum_{i=1}^2 \Vert [\eta_\alpha (Z_m),X_i]\Vert_2 \Vert Z_m\Vert_\infty \left (
2 \Vert \delta_i^*(1\otimes 1)\Vert_2 + 4 \Vert (E_Q\otimes 1)\circ \bar\delta_i\Vert_{M\to L^2(M)}\right).
\end{multline*}

However, since $\{X_i\}_{i=1}^2$ (and so also $\{u_j\}_{j=1}^k$) are in the multiplicative domain of $\eta_\alpha$, we see that $\Vert [\eta_\alpha(Z_m),X_i] \Vert_2 \to 0$ and also $\Vert [\delta_P (\eta_\alpha(Z_m)),u_i] \Vert_2 = \Vert \delta_P (\eta_\alpha ([Z_m,u_i] ))\Vert_2 \to 0$.

By the proof of  \cite[Theorem 3.3]{peterson:propertyT},  since $W^*(X_1, X_2)\subset M$ is a non-amenable subfactor, it follows that $\eta_\alpha$  converges uniformly on its asymptotic commutant.  In particular, $\eta_\alpha$ converges uniformly on $Z_m$.  We may thus take $\alpha\to \infty$ to obtain the desired conclusion.
\end{proof}

We now apply this Proposition to the case of algebras arising from a planar algebra $P$ with $\delta > 1$.  Assume that $P$ is finite-depth of depth $\leq k$.  Graphically, this means: $$
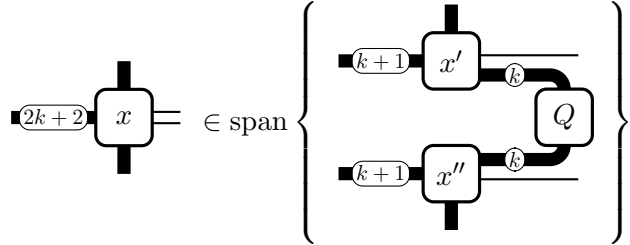
\begin{array}{c}
\begin{tikzpicture}[scale=.75]
 \draw[Box] (0,0) rectangle (1,1); \node at (.5,.5) {$x$};
 \draw[thickline](0,0.5)--node[rcount,scale=.75]{$2k+2$}++(-1.5, 0);
 \draw[thickline](0.5,1)--++(0,.5);
 \draw[thickline](0.5,0)--++(0,-.5);
 \draw[thick](1,0.6)--++(0.5,0);
 \draw[thick](1,0.4)--++(0.5,0);
\end{tikzpicture}
\end{array}
 \in \textrm{span}\left\{
 \begin{array}{c}
\begin{tikzpicture}[scale=.75]
\draw[Box] (0,0) rectangle (1,1); \node at (.5,.5) {$x'$};
 \draw[thickline](0,0.5)--node[rcount,scale=.75]{$k+1$}++(-1.5, 0);
 \draw[thickline](0.5,1)--++(0,.5);
 \draw[thick](1,0.6)--++(1.75,0);
 \draw[Box] (0,-1) rectangle (1,-2); \node at (.5,-1.5) {$x''$};
 \draw[thickline](0,-1.5)--node[rcount,scale=.75]{$k+1$}++(-1.5, 0);
 \draw[thickline](0.5,-2)--++(0,-.5);
 \draw[thick](1,-1.6)--++(1.75,0);
 \draw[Box] (2,0.0) rectangle (3,-1); \node at (2.5,-0.5) {$Q$};
 \draw[thickline](1,0.25) --node[rcount,scale=.75]{$k$}++(1.25,0) arc (90:0:.25);
 \draw[thickline](1,-1.25) --node[rcount,scale=.75]{$k$}++(1.25,0) arc (-90:0:.25);
\end{tikzpicture}
\end{array}\right\}
$$

It follows that the derivation $\partial_{k+1} : Gr_{k+1} \to R_k$ defined
in the proof of Theorem~\ref{thrm:highercomm} is valued in $L^2(Gr_{k+1})
\otimes_{Q} L^2(Gr_{k+1})$, where $Q$ is the algebra generated by all elements of the form
$$
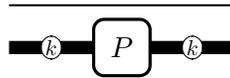
\begin{tikzpicture}[scale=.75]
 \draw[Box] (0,0) rectangle (1,1); \node at (.5,.5) {$P$};
 \draw[thickline](0,0.5)--node[rcount,scale=.75]{$k$}++(-1.5, 0);
 \draw[thickline](1,0.5)--node[rcount,scale=.75]{$k$}++(1.5, 0);
 \draw[thick](-1.5,1.25) -- (2.5, 1.25);
\end{tikzpicture}
$$
We now apply Propostion~\ref{prop:criteriaNonGamma} to $C=Gr_{k+1} \subset M=W^*(Gr_{k+1})$, the finite-dimensional subalgebra $Q$ described above and derivations $\partial_k'$, $\partial_k'' = \partial_k \cdot \Xi''$, where
$$\Xi''_k = \begin{array}{c}
 \begin{tikzpicture}[scale=.5]
\draw[thickline] (0,0) arc (270:90:-0.5 );
\draw[thick] (0,-1.35) arc (-90:-0:1 and .85) ;
\draw[thick] (0,.35) arc (90:0:1 and .85); 

\end{tikzpicture}
 \end{array}\in R_k.
$$
We set
$X_1 = i'(\cup)$, $X_2=i''(\jwcup)$ where $i':Gr_0 \to Gr_{k+1}$, $i'':Gr_1\to Gr_{k+1}$ are the natural inclusions obtained by adding strings at the bottom of a diagram.

It is not hard to see that $\partial_k$ and $\partial_k''$ are co-associative.  It follows from \cite{dabrowski:Fisher} that $(1\otimes E_Q) \circ \partial_k$ and $(1\otimes E_Q)\circ\partial'_k$ are bounded from $M$ to $L^2(M)$.  Since $\partial'$ is a linear combination of $\partial'$ and $\partial$, it follows that it is also bounded.  

Thus Proposition~\ref{prop:criteriaNonGamma} gives us that the asymptotic relative commutant of $X_1,X_2$ in $M_k = W^*(Gr_k)$ for any $k$ strictly bigger than the depth of $P$ is contained in $M_1'\cap M_k$.  Taking compressions, we obtain the same statement for all $k$.

We have thus proved:

\begin{corollary}  Assume that $\Phi^*(\tau)<+\infty$, $\delta > 1$  and that $P$ is finite-depth.  Then for any free ultra-filter $\omega$ and any $k\geq 1$, $W^*(\underbar{$\cup$},\jwcup)'\cap M_{k}^\omega  =  W^*(\underbar{$\cup$},\jwcup)'\cap M_{k}$.  In particular, $W^*(Gr_k)$ does not have property $\Gamma$.
\end{corollary}

\section{The 2-cabled Voiculescu trace}\label{sec:2cabletrace}

The Voiculescu trace was introduced in \cite{gjs1}, and has been fundamental in the series of articles \cite{gjs1, jsw, gjs2, gjsz, cjs}.  In this article, however, it will be the 2-cabled Voiculescu trace that plays a central role.  In this section we will further examine this trace, in particular we show that $M_k^{(\tau)}(\mc P)$ is a II$_1$-factor, and compute its isomorphism class when $\mc P$ is finite-depth.

First consider a graph planar algebra $\mc P^\Gamma$, and recall that we have $Gr_0(\mc P^\Gamma) \subset \ms A_\Gamma$ where $\ms A_{\Gamma}$ is generated by a copy of $\C\Gamma_+$ and operators $\{X_{e,f^o}: e,f \in E_+, \; t(e) = t(f)\}$.  

\begin{proposition}
(a) The 2-cabled Voiculescu trace on $Gr_0(\mc P^\Gamma)$ is the restriction of the $l^\infty(\Gamma_+)$-valued distribution on $\ms A_\Gamma$ for which $(X_{e,f^o})$ form an operator-valued circular family with covariance $\mb E_{l^\infty(\Gamma_+)}[X_{e_1,f_1^o}X_{f_2,e_2^o}] = \delta_{(e_1,f_1^o) = (e_2,f_2^o)}\cdot \mu_{s(f_1)} p_{s(e_1)}$. 
\\ (b) If $\Gamma$ is finite then one can take $X_{e,f^o} = p_{s(e)}C_{e,f^o}p_{s(f)}$ where $C_{e,f^o}$ is a scalar-valued free circular family which is freely independent from $\C\Gamma_+$ with respect to the trace which puts weight $(\sum \mu_w)^{-1} \cdot \mu_v$ on $p_v$.  
\\ (c) The 2-cabled Voiculescu traces  $\tau_k$ are positive.
\end{proposition}

\begin{proof}
The first statement follows from the definition of $\mc P^\Gamma$, see Remark \ref{normalize}.  The second statement is a standard result in free probability on compressed circular systems, see e.g. \cite{nsp}.  The last statement follows from positivity of the conditional expectation associated to these operator-valued semicircular systems.
\end{proof}

\begin{proposition}
If $\delta > 1$, then $M_k^{(\tau)}(\mc P)$ is a II$_1$-factor for each $k \geq 0$.  Moreover, $(M_k^{(\tau)})_{k \geq 0}$ is the Jones tower for the subfactor $M_0^{(\tau)} \subset M_1^{(\tau)}$ and its planar algebra is $\mc P$.
\end{proposition}

\begin{proof}
It is immediate to verify that $\tikz{\draw (0,0) arc(-90:0:.15 and .25); \draw (.5,0) arc(270:180:.15 and .25);}
$
is a conjugate variable for $\tau$.  Thus $\Phi^*(\tau)<\infty$.  Thus we can apply the results of the previous sections to deduce the conclusion of the theorem.
\end{proof}

Following almost verbatium the proof in \cite{gjs2} gives us the identification of the algebras $M_k$; we omit the proof.
\begin{proposition}
If $\mc P$ is finite-depth then $M_k^\tau(\mc P) \simeq L\mb F_{t_k}$ where 
\begin{equation*}
t_k = 1 + \delta^{-2k}I(\delta^2 - 1)
\end{equation*}
with $I$ the global index.
\end{proposition}

\begin{remark}

In \cite{gjs2} it was shown that when $\mc P$ is finite-depth, for the Voiculescu trace we have $M_k \simeq L\mb F_{r_k}$ with $r_k = 1 + 2I\delta^{-2k}(\delta^2 - 1)$ where $I$ is the global index.  In the infinite-depth case, Hartglass \cite{hart} has shown that $M_k \simeq L\mb F_{\infty}$.  

For the 2-cabled Voiculescu trace, we have shown that $M_k \simeq L\mb F_{t_k}$ with $t_k = 1 + \delta^{-2k}I(\delta^2 - 1)$ when $\mc P$ is finite-depth.  Note that these formulas are compatible with the index 2 inclusion of Remark \ref{2cable-remark}.  
\end{remark}

\section{Free Gibbs states and random matrix models}\label{sec:gibbs}

In this section we will construct the free Gibbs state $\tau_V$ on a finite-depth planar algebra $\mc P$, when the potential $V$ is sufficiently close to the quadratic potential $\frac{1}{2}\dcup$.  Such models have already been studied in \cite{gjsz} as well as \cite{nelson}, the main difference is that here we work with a ``2-cabled'' version which leads to slightly different random matrix models.  However we will use essentially the same methods, which go back to earlier work of Guionnet and Maurel-Segala \cite{gms}.  For this reason, we do not give details of proofs (which are word for word the same as used e.g. in \cite{nelson}), but rather show how to modify the appropriate statements to deal with the 2-cabled situation.

We will consider potentials $V = \frac{1}{2}\dcup +  W$ where $W = \sum t_i \cdot W_i$ with $W_i \in P_{n_i}$, $i = 1,\dotsc,k$.  Recall that $\tau_V$ must satisfy the \textit{Schwinger-Dyson equation}:
\begin{equation}\label{eq:SD}
 \tau_V\bigl[(\corner + \ms DW) \wedge a\bigr] = (\tau_V \otimes \tau_V)[\partial(a)], \qquad (a \in Gr_1(\mc P)).
\end{equation}
In the usual (polynomial) case, diagrammatic formulas for $\tau_V$ go back to the fundamental work of 't~Hooft \cite{hooft} and Br\'{e}zin-Itzykson-Parisi-Zuber \cite{bipz}.  This was recently extended to the planar algebra setting in \cite{gjsz}.  Our formula here will be identical to that of \cite{gjsz}, except that we restrict to diagrams in which the strings have been doubled.  Thus at $t = 0$ we recover the 2-cabled Voiculescu trace, instead of the standard Voiculescu trace used there.  

\begin{theorem}
Let $\tau$ be a positive trace satisfying the Schwinger-Dyson equation with some potential $V$.  Then $\Phi^*(\tau)<\infty$.  Thus if $\delta>1$, then $M^{(\tau)}_k$ are factors, and the standard invariant of the inclusion $M_0^{\tau}\subset M_1^\tau$ is $\mc P$.
\end{theorem}
\begin{proof}
All that needs to be pointed out is that, by the Schwinger-Dyson equation, $\mathscr{D}V$ is a conjugate variable for $\tau$, so $\Phi^*(\tau)<\infty$. 
\end{proof}

Define a $\mc P$-functional $\tr_t = (T^{(t)}_m)_{m \geq 1}$ by 
\begin{equation*}
\begin{tikzpicture}[scale=.75]
 \draw[Box] (0,0) rectangle (1,1); \node at (.5,.5) {$P$}; \node[marked, below right, scale=.8] at (1,0) {};
 \draw[thickline] (.5,0) -- ++(0,-.5);
 \node[left] at (-.5,0) {$T^{(t)}_m = \displaystyle \sum_{n_1,\dotsc,n_k = 0}^\infty  \displaystyle \sum_{P \in P(n_1,\dotsc,n_k;m)} \prod_{i=1}^k \frac{t_i^{n_i}}{n_i!}$};
\end{tikzpicture}
\end{equation*}
where $P(n_1,\dotsc,n_k;m)$ is the set of labelled 2-cabled $m$-tangles with $n_1 + \dotsb + n_k$-input discs, $n_i$ of which are labelled by $W_i$ for $1 \leq i \leq k$, which are topologically connected when including the outer disc.  Here 2-cabled means that that the tangle can be obtained by doubling the strings of some other (unshaded) planar tangle.  

Let $P^c(n_1,\dotsc,n_k;m)$ denote the subcollection of tangles which can be obtained by doubling a (unshaded) planar tangle which remains topologically connected after removing the outer disc.  Then it is not hard to see from Definition \ref{planar_cumulants} that the planar algebra cumulants associated to $\tr_t$ as above are given by
\begin{equation*}
\begin{tikzpicture}[scale=.75]
 \draw[Box] (0,0) rectangle (1,1); \node at (.5,.5) {$P$}; \node[marked, below right, scale=.8] at (1,0) {};
 \draw[thickline] (.5,0) -- ++(0,-.5);
 \node[left] at (-.5,0) {$\kappa^{\mc P}_m(t) = \displaystyle \sum_{n_1,\dotsc,n_k = 0}^\infty  \displaystyle \sum_{P \in P^c(n_1,\dotsc,n_k;m)} \prod_{i=1}^k \frac{t_i^{n_i}}{n_i!}$};
\end{tikzpicture}
\end{equation*}

\begin{theorem}
For $|t_i|$ sufficiently small, $\tr_t$ is a positive $\mc P$-trace and is the unique solution of the Schwinger-Dyson equation.
\end{theorem}

\begin{corollary}
Let $\mc A_1,\dotsc, \mc A_m$ be unital, graded $*$-subalgebras of $Gr_0(\mc P)$, and suppose that
\begin{equation*}
\begin{tikzpicture}[thick,scale=.5]
\draw[Box] (0,0) rectangle (2,1); \node at (1,.5) {$a_1$}; \node[marked,scale=.8,below] at (1,0) {};
\draw[Box] (3,0) rectangle (5,1); \node at (4,.5) {$a_2$}; \node[marked,scale=.8,below] at (4,0) {};

\draw (.85,1) arc(180:90:.65) --++(2,0) arc(90:0:.65);
\draw (1.15,1) arc(180:90:.5) --++(1.7,0) arc(90:0:.5);

\draw[medthick] (.5,1) arc(0:180:.375) --++(0,-1.25);
\draw[medthick] (1.5,1) arc(180:0:.375 and .25) --++(0,-1.25);
\draw[medthick] (3.5,1) arc(0:180:.375 and .25) --++(0,-1.25);
\draw[medthick] (4.5,1) arc(180:0:.375 and .25) --++(0,-1.25);

\node[right] at (5.5,.75) {$= 0,$};
\end{tikzpicture}
\end{equation*}
whenever $a_j \in \mc A_{i_j}$, $i_1 \neq i_2$, and the thick lines represent an arbitrary number of parallel strings.  Let $W_j \in \mc A_j$ for $1 \leq j \leq m$ and let $V = \frac{1}{2}\dcup + \sum t_i \cdot W_i$.  Then for $|t_i|$ sufficiently small, $(\mc A_j)_{1 \leq j \leq m}$ are free with respect to $\tau_V$.  In particular, these subalgebras are free with respect to the 2-cabled Voiculescu trace.
\end{corollary}

\begin{proof}
Consider the formula above for the planar algebra cumulants associated to $\tau_V$.  Using the given assumption it is not hard to verify that the hypotheses of Corollary \ref{cumulant_freeness} are satisfied, and the result follows.
\end{proof}

We briefly mention that there is also an associated random matrix model, which creates 2-cabled traces in the large-$N$ limit.  The only difference compared to~\cite{gjsz} is that the blocks of the matrix are labeled by pairs of edges $(e,f^o)$ rather than by a single edge as in~\cite{gjsz}.  We leave the details to the reader.

\section{Free monotone transport}\label{sec:transport}

In this section we will construct the free monotone transport from the 2-cabled Voiculescu trace to a free Gibbs state $\tau_V$ for $V$ sufficiently close to the quadratic potential $\dcup$.  The proof follows the construction of free monotone transport in the polynomial case given in \cite{gsh} and follows the same idea as  in \cite{nelson}.  

\begin{theorem}
Let $\mc P$ be a finite-depth planar algebra.  Let $V=  \frac{1}{2}\dcup +  W$, with $W=\sum t_i W_i$, $W_i \in \mc P$.  Let $\tau$ be the 2-cabled Voiculescu trace, and for $t_i$ small, let $\tau_V$ be unique $\mc P$-trace satisfying the Schwinger-Dyson equation with potential $V$.  Then there exit trace-preserving isomrophisms $C^*(Gr_k,\tau)\cong C^*(Gr_k,\tau_V)$ and $W^*(Gr_k,\tau)\cong W^*(Gr_k,\tau_V)$.
\end{theorem}

We only sketch the proof, which verbatim repeats the argument in \cite{nelson} (except for the use of the slightly different differential calculus adapted to the 2-cabled situation).  

Let $\Gamma$, $\ms B_\Gamma$, $\phi$, $\pi$, $\pi_1$ and $\pi_\boxplus$ be as in \S\ref{sec:embedding-nt}.    By~\cite{nelson:qf}, for small enough $t_i$ there exists a unique state free Gibbs $\phi_V$ on $\ms B_\Gamma$ corresponding to potential $\pi(V)$.  This state has the same modular group as $\phi$ (and is equal to $\phi$ if $t_i=0$ for all $i$).  The composition of $\phi_V$ and $\pi$ give rise to a trace on $Gr_0$ and thus also on all $Gr_k$;  we for now denote this sequence of traces by $\tau'_V$ .    Since $\pi$, $\pi_1$, $\pi_\boxplus$ are equivariant with respect to free differential calculus, it follows that the $\mc P$-trace $\tau'_V$ also satisfies  the Schwinger-Dyson equation \eqref{eq:SD}, and by uniqueness of its solutions, must be equal to $\tau_V$, the free Gibbs $\mc P$-trace on $Gr_k(\mc P)$.

Let $N=\{(e,f^o): t(e)=t(f)\}$.  Let $F\in W^*(\ms A_\Gamma)^N$ be the transport map constructed in \cite{nelson:qf}; thus $F=(F_{(e,f^o)})_{(e,f^o)\in N}$ and the joint law under $\phi_V$ of $F$ is the free quasi-free state $\phi$.  The map $F$ is constructed as a limit (in a certain norm stronger than the operator norm) of maps $F_k$, obtained by applying a certain iterative procedure, which involves the operators $\#$, $\cdot$, $\ms J$, $\ms D$.  The argument in \cite{nelson} goes through verbatim to show that each $F_k$ belongs to $W^*(Gr_1(\mc P))$ (in fact, in a certain analytic subalgebra of this algebra). We thus conclude that $F\in C^*(Gr_1(\mc P))$.  This means that the change of variables formula map $\operatorname{ev}_F$ (see \S\ref{sec:defOfEv})  satisfies
$$
\tau_V(\operatorname{ev}_F (P)) = \tau (P).
$$

This means that $P\mapsto \operatorname{ev}_Y(P)$ extends to a trace-preserving $*$-homomorphism from $C^*(Gr_0(\mc P),\tau)$ to $C^*(Gr_0(\mc P),\tau_V)$.  Because this isomorphism is trace-preserving, it also extends to the associated von Neumann algebras. 

Arguing exactly as in \cite{nelson}, we can show that there is an inverse map $\hat F$ so that $\ev_{\hat{F}}\circ \ev_F = \textrm{id}$; moreover, $\hat F\in C^*(Gr_1(\mc P))$. It follows that the map $  P\mapsto \operatorname{ev}_Y(P) $ is an isomorphism.

\end{document}